\documentclass[a4paper,11pt]{amsart}
\usepackage{geometry}                
\usepackage{graphicx}
\usepackage{amssymb}
\usepackage{xspace}
\usepackage{epstopdf}
\usepackage{tabularx}
\usepackage{enumerate}
\usepackage{floatpag}
\usepackage{geometry}
\usepackage{multirow}
\usepackage{tikz}

\newcommand{\Magma}{{\textsc {Magma}}\xspace}

\newtheorem{thm}{Theorem}[section]

\newtheorem{lem}[thm]{Lemma}
\newtheorem{cor}[thm]{Corollary}
\newtheorem{prop}[thm]{Proposition}

\theoremstyle{definition}
\newtheorem{Definition} [thm] {Definition}
\newtheorem{Remark}  [thm] {Remark}

\newcommand{\Stab}{\mathop{\mathrm{Stab}}}
\newcommand{\Syl}{\mathop{\mathrm{Syl}}}

\DeclareMathOperator{\PSp}{\mathrm{S}}   \DeclareMathOperator{\POmega}{\mathrm{O}}
\DeclareMathOperator{\GL}{\mathrm{GL}}   \DeclareMathOperator{\Lg}{\mathrm{L}}
\DeclareMathOperator{\Ug}{\mathrm{U}}
\DeclareMathOperator{\Alt}{\mathrm{A}}
\DeclareMathOperator{\Sym}{\mathrm{S}}

\newcommand{\Aut}{\mathop{\mathrm{Aut}}}

\newcommand{\Out}{\mathop{\mathrm{Out}}}

\newcommand{\ws}{{\widetilde{\omega}}}

\numberwithin{equation}{section}

\title[Sharp upper bounds]{Sharp upper bounds on the minimal number of elements required to generate a transitive permutation group}
\author{Gareth Tracey}
\address{Mathematical Institute, University of Oxford}
\email{garethtracey1@gmail.com}

\thanks{The author was supported by the Engineering and Physical Sciences Research Council (grant no. EP/T017619/1).}


\begin{document}
\begin{abstract}
The purpose of this paper is to prove that if $G$ is a transitive permutation group of degree $n\geq 2$, then $G$ can be generated by $\lfloor cn/\sqrt{\log{n}}\rfloor$ elements, where $c:=\sqrt{3}/2$. Owing to the transitive group $D_8\circ D_8$ of degree $8$, this upper bound is best possible. Our new result improves a 2015 paper by the author, and makes use of the recent classification of transitive groups of degree $48$.
\end{abstract}

\maketitle

\section{Introduction}
\label{sec:gennos}
For an arbitrary group $G$, let $d(G)$ be the minimal size of a generating set
of $G$. In \cite{MR3812195}, the problem of finding numerical upper bounds for $d(G)$ for an arbitrary transitive permutation group $G$ of degree $n$ is considered. It had already been proved in \cite{Luc} that $d(G)$ is at most $\frac{cn}{\sqrt{\log n}}$ in this case, where $c$ is an absolute constant. This bound is shown to be asymptotically best possible in \cite{KN} (that is, there exists constants $c_1$, $c_2$, and an infinite family $(G_{n_i})_{i=1}^{\infty}$ of transitive groups of degree $n_i$, with $c_1\le \frac{n_i}{d(G_{n_i})\sqrt{\log n_i}}\le c_2$ for all $i$).

In \cite{MR3812195} it is proved that, apart from a finite list of possible exceptions, the bound $d(G)\le \left\lfloor \frac{cn}{\sqrt{\log n}}\right\rfloor$ holds, where $c:=\frac{\sqrt{3}}{2}$. This bound is best possible in the sense that $d(G)= \frac{\frac{\sqrt{3}}{2}n}{\sqrt{\log n}}$ when $G=D_8\circ D_8<\Sym_8$ and $n=8$.

In this paper, we remove the possible exceptions listed in \cite[Theorem 1.1]{MR3812195}. More precisely, we prove: 
\begin{thm}\label{thm:transgens}
Let $G$ be a transitive permutation group of degree $n$. Then 
$$d(G)\le \left\lfloor\frac{cn}{\sqrt{\log{n}}}\right\rfloor$$
where $c:=\frac{\sqrt{3}}{2}$.
\end{thm}

There are a number of steps in the proof of Theorem \ref{thm:transgens}. For this reason, we will spend the next few paragraphs outlining our general strategy for the proof.

First, by \cite[Theorem 5.3]{MR3812195}, we only need to prove Theorem \ref{thm:transgens} when $G$ is imprimitive with minimal block size $2$, and $n$ has the form $n=2^x3^y5$ with either $y=0$ and $17\le x\le 26$; or $y=1$ and $15\le x\le 35$. Thus, in particular, $G$ may be viewed as a subgroup in a wreath product $2\wr G^{\Sigma}$, where $\Sigma$ is a set of blocks for $G$ of size $2$. It follows that $d(G)\le d_{G^{\Sigma}}(M)+d(G^{\Sigma})$, where $M$ is the intersection of $G$ with the base group of the wreath product, and $d_{G^{\Sigma}}(M)$ is the minimal number of elements required to generate $M$ as a $G^{\Sigma}$-module.

Of course, $d(G^{\Sigma})$ can be bounded using induction. The bulk of this section will therefore be taken up with finding upper bounds on $d_{G^{\Sigma}}(M)$. Since $d_{G^{\Sigma}}(M)\le d_H(M)$ for any subgroup $H$ of $G^{\Sigma}$, the strategy of the third author in \cite{MR3812195} in this case involved replacing $G^{\Sigma}$ by a convenient subgroup $H$ of $G^{\Sigma}$, and then deriving upper bounds on $d_H(M)$, usually in terms of the lengths of the $H$-orbits in $\Sigma$. This approach turns out to be particularly fruitful when $H$ is chosen to be a soluble transitive subgroup of $G^{\Sigma}$, whenever such a subgroup exists. When $G^{\Sigma}$ does not contain a soluble transitive subgroup, however, the analysis becomes much more complicated. This leads to less sharp bounds, and ultimately, the omitted cases in \cite[Theorem 1.1]{MR3812195}.

Our approach in this section involves a careful analysis of the orbit lengths of soluble subgroups in a minimal transitive insoluble subgroup of $G$, building on the work in \cite{MR3812195} in the case $n=2^x3$. 

The layout of the paper is as follows: First, Section \ref{sec:Dir} contains a necessary discussion of subdirect products of finite groups.
Then, in Section \ref{sec:MinTrans} we prove a general theorem concerning minimal transitive groups of degree $2^x3^y5^z$, with $0\le y,z\le 1$. Finally, this information will allow us to prove Theorem \ref{thm:transgens}, and we do so in Section \ref{sec:proof}.

\vspace{0.25cm}
\noindent\textbf{Acknowledgements:} It is a pleasure to thank Derek Holt and Gordon Royle for helpful conversations and suggestions.

\section{Subdirect products of finite groups}\label{sec:Dir}
As mentioned above, we begin preparations toward the proof of Theorem \ref{thm:transgens} with some straightforward but necessary lemmas concerning subdirect products of finite groups. First, we make the following remark concerning our notation and terminology in direct products of finite groups:
\begin{Remark}
For a group $T$ we will write $T^e$ for the direct product of $e$ copies of $T$. More generally, let $G_1,\hdots,G_r$ be finite groups, and consider the direct product $G:=G_1^{e_1}\times\hdots\times G_r^{e_r}$. Fix subgroups $H_i\le G_i$, for $1\le i\le r$, and set $e:=\sum_{i=1}^{r}k_i$. Then we write $H_1^{e_1}\times\hdots\times H_r^{e_r}$ for the subgroup $\{(x_1,\hdots,x_e)\text{ : }x_j\in H_i\text{ for }f_{i-1}+1\le j\le f_{i}\}$ of $G$, where $f_0:=0$, and $f_i:=\sum_{k=1}^{i}e_k$ for $i>0$. 

\noindent Also, the subgroup $\Sym_{e_1}\times\hdots\times\Sym_{e_r}$ of $\Sym_e$ (in its natural intransitive action) acts via automorphisms on $G$ (by permutation of coordinates). Thus we can, and do, speak of ($\Sym_{e_1}\times\hdots\times\Sym_{e_r}$)-conjugates of subgroups of $G$. This will be very useful for avoiding cumbersome notation.

\noindent Finally, for a group $T$ we call a subgroup $G$ of $T^e$ a \emph{diagonal subgroup} of $T^e$ if there exists automorphisms $\alpha_i$ of $T$ such that $G=\{(t^{\alpha_1},t^{\alpha_2},\hdots,t^{\alpha_r})\text{ : }t\in T\}$. 
\end{Remark}

Suppose now that $G$ is a subgroup in a direct product $G_1\times\hdots\times G_e$ of groups $G_i$, and write $\pi_i:G\rightarrow G_i$ for the $i$th coordinate projections. Then $G$ is a \emph{subdirect product} in $G_1\times\hdots\times G_e$ if $G\pi_i=G_i$ for all $i$. We also introduce the following non-standard definition:
\begin{Definition}\label{Def:Subd}
Let $G_i$ be as above, and let $n$ be a positive integer. We say that $G$ is \emph{a subdirect product of the form} $G=\frac{1}{n}(G_1\times\hdots\times G_e)$ if each of the following holds:
\begin{enumerate}[(a)]
    \item $G$ is a subdirect product in $G_1\times\hdots\times G_e$;
    \item $G$ contains $[G_1,G_1]\times\hdots\times [G_e,G_e]$; and
    \item $G$ has index $n$ in $G_1\times\hdots\times G_e$.
\end{enumerate}
\end{Definition}

The following is a collection of easy facts concerning subdirect products in direct products of finite groups:
\begin{lem}\label{lem:Subd}
Let $G_1,\hdots,G_e$ be finite groups, and $\pi_i:G_1\times\hdots\times G_e\rightarrow G_i$ be as above, and let $K_i$ be the intersection of $G_i$ with the $i$th coordinate subgroup of $G_1\times\hdots\times G_e$. Suppose that $G$ is a subdirect product in $G_1\times\hdots\times G_e$. Then each of the following holds:
\begin{enumerate}[(i)]
    \item $K_i\pi_i\unlhd G_i$ for all $i$.
    \item If $e=2$ and $G_1$ and $G_2$ have no common nontrivial homomorphic images, then $G=G_1\times G_2$.
    \item Let $e=\sum_{i=1}^r e_i$ be a positive partition of $e$, and define $f_0:=0$, and $f_i:=\sum_{k=1}^{i}e_k$ for $1\le i\le r$. Suppose that for each $1\le i\le r$, there exists a finite simple group $T_i$ such that $G_j$ is isomorphic to $T_i$ for all $f_{i-1}+1\le j\le f_{i}$. Then for each $1\le i\le r$ there exists a partition $\sum_{k=1}^{t_i} d_{i,k}$ of $e_i$ and diagonal subgroups $D_{i,k}$ of $T_i^{d_{i,k}}$ such that $G$ is $(\Sym_{e_1}\times\hdots\times\Sym_{e_r})$-conjugate to $(D_{1,1}\times \hdots\times D_{1,t_1})\times \hdots\times (D_{r,1}\times\hdots\times D_{r,t_r})$.
    \item Let $T$ be a nonabelian simple group, and suppose that $|G_i|<|T|$ for all $i$. Then $G$ has no homomorphic image isomorphic to $T$. 
    \item Fix $1< f< e$. Suppose that $G_i$ is nonabelian simple for $1\le i\le f$, and that $|G_j|<|G_i|$ for all $1\le i\le f$ and $f+1\le j\le e$. Assume also that $|G_i|$ does not divide the index of $G$ in $G_1\times\hdots\times G_e$ for any $1\le i\le f$. Then $G=G_1\times\hdots\times G_f\times L$, where $L$ is a subdirect product in $G_{f+1}\times\hdots\times G_e$.   
\end{enumerate}
\end{lem}
\begin{proof}
Part (i) is clear, while part (ii) follows from the fact that if $e=2$, then 
$$\frac{G_1}{K_1\pi_1}\cong \frac{G_2}{K_2\pi_2}.$$
Part (iii) is well-known (for example, see \cite[Theorem 4.16(iii)]{PS}), so we just need to prove (iv) and (v). 

We begin with (iv). So assume that $T$ is a nonabelian finite simple group, and that $|G_i|<|T|$ for all $i$. We prove the claim in (iv) by induction on $e$, with the case $e=1$ being trivial. So assume that $e>2$, and write $\hat{\pi}$ for the projection $\hat{\pi}:G\rightarrow G_2\times\hdots\times G_e$. Suppose that $N$ is a normal subgroup of $G$ with $G/N\cong T$. Then since $G\hat{\pi}/N\hat{\pi}$ is both a homomorphic image of $G/N\cong T$, and the subdirect product $G\hat{\pi}\le G_2\times\hdots\times G_e$, the inductive hypothesis implies that $N\hat{\pi}=G\hat{\pi}$. Then since $K_1$ is the kernel of $\hat{\pi}$, we have $|G|=|K_1||G\hat{\pi}|$ and $|N|=|K_1\cap N||G\hat{\pi}|$. Hence, $|G_1|\geq |K_1|\geq |K_1/K_1\cap N|=|T|$ -- a contradiction. This proves (iv).  

Finally, assume that the hypotheses in (v) hold. Then $G$ may be viewed as a subdirect product in $R\times L$, where $R$ is a subdirect product in $G_1\times\hdots\times G_f$, and $L$ is a subdirect product in $G_{f+1}\times\hdots\times G_e$. It then follows from (ii), (iii), and (iv) that $G=R\times L$. Moreover, since $|G_1\times\hdots\times G_f:R|$ divides $|G_1\times\hdots\times G_e:G|$, we must have $R=G_1\times\hdots\times G_f$, by (iii). This completes the proof of (v), whence the lemma.
\end{proof}

\section{Minimal transitive groups of degree $2^x3^y5^z$}\label{sec:MinTrans}
In this subsection, we investigate the structure of minimal transitive groups of degree $2^x3^y5^z$, with $0\le y,z\le 1$. Our main result reads as follows:
\begin{thm}\label{thm:MinTransTheorem} Let $G$ be an insoluble minimal transitive permutation group of degree $n=2^x3^y5^z$, with $0\le y,z\le 1$. Then:
\begin{enumerate}[(i)]
\item $G$ has at most $y+z$ nonabelian chief factors;
\item each nonabelian chief factor of $G$ is isomorphic to a direct product of copies of one of $\Alt_5$; $\Alt_6$; $\PSp_4(4)$; $\POmega^+_8(2)$; $\Alt_8$;
 $\Alt_{16}$; or
$\Lg_2(q)$, where $q$ has the form $q=2^{x_1}3^{y_1}5^{z_1}-1$ with $0\le y_1,z_1\le 1$; 
\item if $G$ has two nonabelian chief factors, then they are isomorphic to $T_1^{e_1}$ and $\Lg_2(p)^{e_2}$, where $T_1\in\{\Alt_5,\Alt_6\}$, and $p$ is a Mersenne prime; and
\item $G$ has a soluble subgroup $F$ whose orbit lengths are given in Table \ref{tab:OrbitSizes}.
\end{enumerate}\end{thm}

Now, Table \ref{tab:OrbitSizes} requires some explanation: In Theorem \ref{thm:MinTransTheorem}, we prove that a minimal transitive permutation group of degree $2^x3^y5^z$ with $0\le y,z\le 1$, has at most two nonabelian chief factors, which have to be isomorphic to direct powers of one of $\Alt_5$; $\Alt_6$; $\Alt_8$;
 $\Alt_{16}$; or
$\Lg_2(q)$, where $q$ has the form $q=2^{x_1}3^{y_1}5^{z_1}-1$. The first column of Table \ref{tab:OrbitSizes} will be a group $T^e$, where $T$ is one of these simple groups. This means that if $G$ is the minimal transitive group under consideration, then $G/R(G)$ has a minimal normal subgroup $L/R(G)$ which is isomorphic to a direct product of $e$ copies of $T$ (recall that $R(G)$ is the soluble radical of $G$). The second column gives the degree of $G$ as a permutation group, while the third column gives the orbit lengths of a soluble subgroup $F$ of $G$. Moreover, these orbit lengths will be given in one of the following ways:
\begin{description}
    \item[Rows 23--30] In these cases, the orbit lengths are given in the form of a list with entries $\sum_{i=1}^fk_i\times l_i$, with the entry $\sum_{i=1}^{f}k_i\times l_i$ meaning that $F$ has $k_i$ orbits of size $l_i$ for each $1\le i\le f$.
    \item[Rows 1--11 and 13--21] In these cases, the orbit lengths are given in the form $\alpha X_i$, $\beta Y_j$. Here, $\{X_i\}_i$ and $\{Y_j\}_j$ are positive unordered partitions whose sum is given in the fourth column of the table. For example, if $F$ is as in row 1, then $\{X_i\}_i$ is a positive unordered partition of $4$ (there are five such partitions) and $\{Y_j\}_j$ is a positive unordered partition of $2$ (there are two such partitions). Thus, there are ten possibilities for the orbit lengths of $F$. For instance, the orbit sizes ($40a$, $100a$, $100a$) correspond to the partitions $4=4$ and $2=1+1$ of $4$ and $2$, respectively.
    \item [Rows 12 and 22] The orbits lengths in these cases are given in the form $\sum_{i=0}^{e_2}\sum_{j,k}C^{(i)}_j\times 2^b3p^i5X^{(i,j)}_k$, $0\le i\le e_2$. By this notation we mean that $(C^{(i)}_j)_j$ is a positive ordered partition of a particular number $C$ for each $i$; $(X^{(i,j)}_k)_k$ is a positive ordered partition of a particular number $X$ for each $i,j$; and $F$ has $C^{(i)}_j$ orbits of size $2^b3p^i5X^{(i,j)}_k$ for each $i$, $j$, $k$. The numbers $C$, $X$, $b$, and $p$ are given in the fourth column of the table (in fact, $C:=\binom{e_2}{i}$ in each case).
\end{description}

As an example, assume that $G$ is a minimal transitive group of degree $n=2^{15}15$. If a minimal normal subgroup of $G$ is isomorphic to a direct product of copies of $\Alt_5$, then the first section of Table \ref{tab:OrbitSizes} gives us the orbit lengths of a soluble subgroup $F$ of $G$. For example, if $F$ is as in the first line of the table, then $F$ has orbit lengths $2^{12}5X_i$, $2^{13}25Y_j$, where $(X_i)_i\in\{(4),(3,1),(2,2),(1,1,1,1)\}$, and $(Y_j)_j\in\{(2),(1,1)\}$. Thus, the possible $F$-orbit lengths in this case are: ($2^{14}5$, $2^{14}25$); or  ($2^{14}5$, $2^{13}25$, $2^{13}25$); or ($2^{12}15$, $2^{12}5$, $2^{14}25$); or ($2^{12}15$, $2^{12}5$, $2^{13}25$, $2^{13}25$); or ($2^{13}5$, $2^{13}5$, $2^{14}25$); or ($2^{13}5$, $2^{13}5$, $2^{13}25$, $2^{13}25$); or ($2^{12}5$, $2^{12}5$, $2^{12}5$, $2^{12}5$, $2^{14}25$); or ($2^{12}5$, $2^{12}5$, $2^{12}5$, $2^{12}5$, $2^{13}25$, $2^{13}25$). The rest of the possibilities for the orbit lengths of $F$ can be computed in an entirely analogous way. (Although these computations seem tedious, they can be done very quickly using a computer).

\begin{table}[]
    \centering\thisfloatpagestyle{empty}
    \begin{tabular}{p{0.9cm}|p{2.5cm}|p{6.05cm}|p{6.5cm}}
      $\overline{L}$   & degree of $G$ & \small{orbit lengths of a soluble subgroup $F$} & notes\\
      \hline
     $\Alt_5^e$  & $240a$ & $10aX_i$, $100aY_j$ & $(\sum_i X_i,\sum_j Y_j)=(4,2)$\\
     & $120a$ & $5aX_i$, $50aY_j$ & $(\sum_i X_i,\sum_j Y_j)=(4,2)$\\
     & $120a$ & $10aX_i$, $50aY_j$ & $(\sum_i X_i,\sum_j Y_j)=(2,2)$\\
     & $60a$ & $5aX_i$, $25aY_j$ & $(\sum_i X_i,\sum_j Y_j)=(2,2)$\\ 
     & $60a$ & $5aX_i$, $50aY_j$ & $(\sum_i X_i,\sum_j Y_j)=(2,1)$\\
     & $30a$ & $5aX_i$, $25aY_j$ & $(\sum_i X_i,\sum_j Y_j)=(1,1)$\\
     & $2^b10a$ & $2^b5aX_i$ & $\sum_i X_i=2$; $b\in\{0,1\}$; $G/L$ soluble\\
    & $60a$ & $10a$, $50a$ & \\
    & $30a$ & $5aX_i$, $10aY_j$ & $(\sum_i X_i,\sum_j Y_j)=(2,2)$\\ 
    &  $120a$ & $5aX_i$, $25aY_j$ & $(\sum_i X_i,\sum_j Y_j)=(4,4)$\\ 
    &  $60a$ & $5aX_i$, $25aY_j$ & $(\sum_i X_i,\sum_j Y_j)=(2,2)$\\
    &  \small{$2^b(p+1)^{e_2}30a$}      & $\sum_{i=0}^{e_2}\sum_{j,k}C^{(i)}_j\times 2^b3p^i5aX^{(i,j)}_k$ & $\sum_k C^{(i)}_j=\binom{e_2}{i}$ and
    $\sum_j X^{(i,j)}_k=2$ for each $i,j$; $b\in\{0,1\}$; $G/L$ insoluble; $p$ a Mersenne prime; $1\le e_2\le e$ \\
    \hline
    $\Alt_6^e$  & $240a$ & $10aX_i$, $50aY_j$ & $(\sum_i X_i,\sum_j Y_j)=(4,4)$\\
     & $120a$ & $5aX_i$, $25aY_j$ & $(\sum_i X_i,\sum_j Y_j)=(4,4)$\\
     & $60a$ & $5aX_i$, $25aY_j$ & $(\sum_i X_i,\sum_j Y_j)=(2,2)$\\
     & $120a$ & $10aX_i$ & $\sum_i X_i=12$\\
     & $15a$ & $5aX_i$ & $\sum_i X_i=3$\\
    & $40a$ & $10aX_i$ & $\sum_i X_i=4$; $G/L$ soluble\\
    & $2^b5a$ & $5aX_i$ & $\sum_i X_i=2^b$; $b\in\{1,2\}$; $G/L$ soluble\\
        &  $60a$ & $5aX_i$, $10aY_j$ & $(\sum_i X_i,\sum_j Y_j)=(4,4)$\\ 
    &  $30a$ & $5aX_i$, $10aY_j$ & $(\sum_i X_i,\sum_j Y_j)=(2,2)$\\
        &  \small{$2^bX(p+1)^{e_2}15a$}      & $\sum_{i=0}^{e_2}\sum_{j,k}C^{(i)}_j\times 2^b3p^i5aX^{(i,j)}_k$ & $\sum_j C^{(i)}_j=\binom{e_2}{i}$, 
    $\sum_k X^{(i,j)}_k=X$ for each $i,j$; $(X,b)\in\{(4,1),(4,0),(2,0)\}$; \small{$G/L$ insoluble; $p$ a Mersenne prime; $1\le e_2\le e$}\\
   \hline
    $T^e$  & $2^b120a$ & $2^b24a+2^b96a$ & $b\in\{0,3\}$, $T\in\{\PSp_4(4),\POmega^+_8(2)\}$\\
      \hline
      $\Alt_8^e$  & $8^{e_2}960a$ & $\sum_{j=1}^3\sum_{i=0}^{e_2}as_j\binom{e_2}{i}\times 7^il_j$ & $0\le e_2\le e$; $1\le j\le 3$, where $(s_1,l_1)=(1,1)$, $(s_2,l_2)=(11,7)$, and $(s_3,l_3)=(42,21)$.\\
      & $8^{e_2}120a$ & $\sum_{j=1}^3\sum_{i=0}^{e_2}as_j\binom{e_2}{i}\times 7^il_j$ & $0\le e_2\le e$; $1\le j\le 3$, where $(s_1,l_1)=(1,1)$, $(s_2,l_2)=(5,7)$, and $(s_3,l_3)=(4,21)$.\\
      & $8^{e_2}15a$ & $\sum_{j=1}^2\sum_{i=0}^{e_2}as_j\binom{e_2}{i}\times 7^il_j$ & $0\le e_2\le e$; $1\le j\le 2$, where $(s_1,l_1)=(1,1)$ and $(s_2,l_2)=(2,7)$.\\
      \hline
      $\Alt_{16}^e$  & $15\frac{16^{e_2+1}}{2^b}a$ & $\sum_{i=0}^{e_2}a\frac{16}{2^b}\binom{e_2}{i}\times 15^{i+1}$ & $b\in\{0,1\}$, $0\le e_2\le e$\\
      \hline
      $\Lg_2(q)^e$  & $15a(q+1)^{e_2+2}$ & $\sum_{i=0}^{e_2}a\binom{e_2}{i}\times 15q^{i}+2a\binom{e_2}{i}\times 15q^{i+1}$ & $q$ a Mersenne prime, $0\le e_2\le e$\\
      & $3^{y_0}5^{z_0}a(q+1)^{e_2}$ & $\sum_{i=0}^{e_2}a\binom{e_2}{i}\times 3^{y_0}5^{z_0}q^{i}$ & $q$ a Mersenne prime, $1\le e_2\le e$\\
    & $2^{x'}15a$ & $aq3^{1-y'}5^{1-z'}+a 3^{1-y'}5^{1-z'}$ & $q=2^{x'}3^{y'}5^{z'}-1$, $\{y',z'\}=\{1,0\}$\\
      \hline
      \end{tabular}
      \vspace{0.15cm}
    \caption{Orbit lengths of a certain soluble subgroup of a minimal transitive group $G$ with $\overline{L}:=L/R(G)$ a  minimal normal subgroup of $G/R(G)$.}
    \label{tab:OrbitSizes}
\end{table}

\begin{Remark}\label{rem:MinTransRemark}
Let $G$ be as in the statement of Theorem \ref{thm:MinTransTheorem}, and suppose that we have proved the theorem for all degrees less than $n$. Let $K$ be the kernel of the action of $G$ on a block system $\Sigma=\{\Delta^g\text{ : }g\in G\}$ for $G$ in $\POmega$. Then $G^{\Sigma}\cong G/K$ is minimal transitive. Thus, we may choose a soluble subgroup $F/K$ of $G/K\cong G^{\Sigma}$ satisfying Theorem \ref{thm:MinTransTheorem}(iv) with $G$ replaced by the minimal transitive group $G^{\Sigma}$

If $K$ is soluble and acts transitively on $\Delta$, then $F$ is soluble and the $F$-orbits have lengths $|\Delta||\Sigma_i|$, as $\Sigma_i$ runs over the $F^{\Sigma}$-orbits. Suppose that the $\Sigma_i$ are as in row k in Table \ref{tab:OrbitSizes}. Then since $|\Delta|$ divides $n$, it is easy to see that by replacing $a$ by $|\Delta|a$, $F\le G$ has the required orbit lengths in order to fulfil the conclusion of (iv). (For example, if $F/K$ is as in row 2 of Table \ref{tab:OrbitSizes}, then $G$ has degree $120|\Delta|a$, and $F$ is soluble with orbit sizes $5|\Delta|aX_i$, $50|\Delta|aY_j$.) This is a useful inductive tool which we will use throughout the section.
\end{Remark}



We begin preparations toward the proof of Theorem \ref{thm:MinTransTheorem} with a reduction lemma:
\begin{lem}\label{lem:RedLemma}
Let $G$ be a counterexample to Theorem \ref{thm:MinTransTheorem} of minimal degree. Then all minimal normal subgroups of $G$ are nonabelian. Moreover, if $L$ is a minimal normal subgroup of $G$ and $K$ denotes the kernel of the action of $G$ on the set of $L$-orbits, then $K/L$ is soluble. 
\end{lem}
\begin{proof}
Let $L$ be a minimal normal subgroup of $G$. Also, let $\Delta$ be an $L$-orbit, and let $\Sigma:=\{\Delta^g\text{ : }g\in G\}$ be the set of $L$-orbits in $[n]$, so that $G^{\Sigma}\cong G/K$. Then $n=|\Delta||\Sigma|$, and $G^{\Sigma}$ is minimal transitive. The minimality of $G$ as a counterexample then implies that $G^{\Sigma}$ satisfies the conclusion of Theorem \ref{thm:MinTransTheorem}. 

Let $E/L$ be a subgroup of $G/L$ which is minimal with the property that $(E/L)(K/L)=G/L$. Then $(E/L)\cap (K/L)\le \Phi(E/L)$, so $(E/L)\cap (K/L)$ is soluble. But also, $EK=G$. Since $E$ contains $L$ and $E^{\Sigma}=G^{\Sigma}$ is transitive, we deduce that $E$ is transitive, whence $E=G$, since $G$ is minimal transitive. Thus, $K/L$ is soluble. This proves the second part of the lemma. 

Finally, assume that $L$ is abelian. Then $K$ is soluble. It follows that the set of nonabelian chief factors of $G$ is same as the set of nonabelian chief factors of $G/K$. Thus, $G=G^{\POmega}$ satisfies (i), (ii), and (iii) in Theorem \ref{thm:MinTransTheorem}. Moreover, since the theorem holds in $G^{\Sigma}\cong G/K$, we may choose a soluble subgroup $F/K$ of $G/K$ satisfying Theorem \ref{thm:MinTransTheorem}(iv). Then $F$ is soluble, since $K$ is soluble, and the $F$-orbits in $[n]$ have lengths $|\Delta||\Sigma_i|$, as $\Sigma_i$ runs over the $F^{\Sigma}$-orbits in $\Sigma$. Then $F$ is a subgroup of $G$ satisfying the conclusion of Theorem \ref{thm:MinTransTheorem}(iv) (see Remark \ref{rem:MinTransRemark}). Thus, (i)--(iv) in Theorem \ref{thm:MinTransTheorem} hold in $G$, contradicting the assumption that $G$ is a counterexample.
\end{proof}

By Lemma \ref{lem:RedLemma}, a minimal normal subgroup $L$ of a minimal counterexample to Theorem \ref{thm:MinTransTheorem} must be a direct product of nonabelian simple groups. Since the subnormal subgroups of a transitive permutation group of degree $n$ have orbit lengths dividing $n$, this motivates us to study the pairs ($T$, $H$) where $T$ is a nonabelian simple group, $H$ is a proper subgroup of $T$, and $|T:H|$ has the form $2^x3^y5^z$, for some $0\le y,z\le 1$. 




Our main tool in this regard is a paper of Liebeck, Praeger, and Saxl \cite{LPS}, where the finite simple groups $T$ with a maximal subgroup $M$ satisfying $\pi(|M|)=\pi(|T|)$ are classified. In order to use this result, we will first need to determine the finite simple groups $T$ with a subgroup $H$ of index $2^x3^y5^z$ as above, such that $\pi(|M|)\neq \pi(|T|)$ for some maximal subgroup $M<T$ containing $H$. 
\begin{prop}\label{prop:nonPi}
Let $T$ be a nonabelian finite simple group, and suppose that $T$ has a subgroup $H$ of index $2^{x}3^{y}5^z$, with $0\le y,z\le 1$. Let $M$ be a maximal subgroup of $T$ containing $H$, and assume that either $\pi(|T|)\neq \pi(|M|)$, or that $T$ is an alternating group of degree less than $10$. Then one of the following holds:\begin{enumerate}[(i)]
    \item $T=\Alt_n$ is an alternating group and either
    \begin{enumerate}[(a)]
        \item $n=5$ and $H$ is any subgroup of $T$;
        \item $n=6$ and $|H|=2^i3$, for some $i\le 3$;
        \item $n=7$, and either $H=\GL_3(2)$ (two classes) or $H=7:3< \GL_3(2)$; 
        \item $n=8$ and either $H=2^3:\GL_3(2)$ (two classes); or $H=2^3: (7:3)< 2^3:\GL_3(2)$ (two classes); or $H\cong \GL_3(2)< 2^3:\GL_3(2)$ (three classes); or $H\cong 7:3< 2^3:\GL_3(2)$; or 
        \item $n=9$ and $H=\Lg_2(8).3$ (two classes).
    \end{enumerate}
    \item $T=\Lg_n(q)$ and either:
    \begin{enumerate}
        \item $n=2$, $q$ has the form $q=2^{x}3^{y_1}5^{z_1}-1$, and $H$ is a subgroup of a maximal parabolic subgroup $M$ of $T$, with $|M:H|=3^{y-y_1}5^{z-z_1}$; or
        \item $(n,q)=(4,3)$, and $H=3^3:\Lg_3(3)$ (two classes).
    \end{enumerate}
    \item $T=\PSp_4(3)$ with $|T:H|=40$ (two classes); or $T=\PSp_6(2)$, and $|T:H|\in\{120,960\}$.
\end{enumerate}
\end{prop}
\begin{proof}


Assume first that $T= \Alt_n$. 
We claim that in order for $|T:M|$ to have the form $|T:M|=2^{x_1}3^{y_1}5^{z_1}$ with $0\le y_1,z_1\le 1$ and $\pi(|T|)\neq \pi(|M|)$, we must have $n\le 9$.
To prove this, we first list the the possibilities for $M$:
\begin{itemize}
    \item $M=(\Alt_k\times \Alt_{n-k}).2$, for some $k\le \frac{n}{2}$.
    \item $M=\Alt_n\cap (\Sym_r\wr \Sym_s)$, where $r,s>1$ and $rs=n$. 
    \item $M=\Alt_n\cap X$, with $X\le\Sym_n$ primitive in its natural action.
\end{itemize}
In particular, if $|M|$ is odd, then $M$ must be as in the third listed case with $n$ an odd prime, $\frac{n-1}{2}$ odd, and $M=C_n\rtimes C_{\frac{n-1}{2}}$. However, in this case $|T:M|=(n-2)!$, which has the form $|T:M|=2^{x_1}3^{y_1}5^{z_1}$ with $0\le y_1,z_1\le 1$ if and only if $n\le 7$.  
So we may assume that $|M|$ is even. But then either $|T|_3=3$ or $|T|_5=5$, since $\pi(|T|)\neq \pi(|M|)$. Thus $n\le 9$, as claimed. We can now use direct computation to quickly determine the possibilities for $H$. 

So we may assume that $T$ is not an alternating group. If $T\in\{\Lg_2(8),\Lg_3(3),\Ug_3(3),\PSp_4(8)\}$ then direct computation quickly shows that $T$ has no maximal subgroups $M$ with index of the form $|T:M|=2^{x_1}3^{y_1}5^{z_1}$ with $0\le y_1,z_1\le 1$. So we may assume that $T\not\in\{\Lg_2(8),\Lg_3(3),\Ug_3(3),\\ \PSp_4(8)\}$.

Suppose next that $T=\Lg_2(p)$ with $p=2^x-1$ a Mersenne prime. Then clearly $M$ must contain a Sylow $p$-subgroup of $T$. Hence, $M$ is a maximal parabolic subgroup of $T$; $|T:M|=p+1=2^x$ and $|M:H|=3^y5^z$ (since $|M|$ is odd). 
So assume also that $T$ is not of the form $T=\Lg_2(p)$ with $p$ a Mersenne prime.
Then by \cite[Corollary 6]{LPS}, there exists a set $\Pi$ of odd prime divisors of $|T|$ such that $\Pi$ intersects $\pi(|T|)\setminus \pi(|M|)\subseteq \{2,3,5\}$ non-trivially. Moreover, the proof of  \cite[Corollary 6]{LPS} shows that either $T$ is in \cite[Table 10.1]{LPS} and $\Pi$ is as in the third column of \cite[Table 10.1]{LPS} (ignoring the third row for $T\in\{\PSp_n(q),\POmega_{n}(q)\}$ and the second row for $T=\POmega^-_n(q)$); or $T$ is in \cite[Table 10.3]{LPS} and $\Pi$ is as in the third column of \cite[Table 10.3]{LPS}; or $T$ is in \cite[Table 10.4]{LPS} and $\Pi$ is as in the second column of \cite[Table 10.4]{LPS} (except when $T=\Ug_5(2)$, where $\Pi:=\{3,5,11\}$); or $T$ is in \cite[Table 10.5]{LPS} and $\Pi$ is as in the second column of \cite[Table 10.5]{LPS} (except when $T=G_2(q)$, where $\Pi:=\{p,7,13\},\{5,7,13\}$, or $\{7,13\}$ according to whether $q\geq 5$, $q=4$, or $q=3$, respectively); or $T$ is in \cite[Table 10.6]{LPS} and $\Pi$ is as in the second column of \cite[Table 10.6]{LPS} (except when $T=M_{11},M_{12}$, or $M_{24}$, where $\Pi:=\{3,11\},\{3,11\}$, or $\{7,23\}$, respectively). Fix $r$ to be a prime in $\Pi\cap (\pi(|T|)\setminus \pi(|M|))$, so that $r\in\{3,5\}$. This implies in particular that $|T|_r=r$.

By \cite[Corollary 6]{LPS} and inspection of simple group orders, the case $r=3$ occurs only if $T=\Lg_2(p)$ with $p$ of the form $p=2^{x_1}3^{y_1}-1$, with $0\le y_1\le 1$. In this case, one can see from \cite[Table 8.1]{BHRD} that $M$ is a maximal parabolic subgroup of $T$. Since $2^{x_1}$ is the largest power of $2$ dividing $|T|$ in this case, we deduce that $x=x_1$ and that $H$ must have index dividing $3^{y-y_1}5^z$ in $M$.  

So we may assume, for the remainder of the proof, that $r=5$.
In what follows, for a prime power $q$ and a positive integer $n$ we will write $q_n$ for an arbitrary primitive prime divisor of $q^n-1$. By Zsigmondy's theorem these exists for $n>1$ and $(q,n)\neq (2,6)$. By Fermat's Little Theorem we have
\begin{align}\label{lab:ppd}
\text{$n\mid q_n-1$.}   
\end{align}

Suppose first that $G$ is a finite simple classical group, and write $T=X_n(q)$, where $X\in\{\Lg,\Ug,\PSp,\POmega^{\pm},\POmega\}$, and $q=p^f$ with $p$ prime. Then using (\ref{lab:ppd})
we can deduce from \cite[Tables 10.1--10.6]{LPS} that $T\in\{\Lg_n(q)\text{ }(n\le 5),\Ug_n(q)\text{ }(n\le 6),\PSp_n(q)\text{ }(n\le 8),\POmega^{\pm}_{n}(q)\text{ } (n\le 10),\POmega_n(q)\text{ }(n\le 11)\}$. Below, we examine these remaining cases. 

Suppose first that $T=\Lg_n(q)$, with $n\le 5$. If $n\in\{4,5\}$, then (\ref{lab:ppd}) and \cite[Tables 10.1 and 10.3]{LPS} imply that either $5$ is a primitive prime divisor of $q^4-1$ or $p=5$. If $n\in\{2,3\}$, then $5$ is a primitive prime divisor of $q^2-1$. Suppose first that $M$ is parabolic. Then we see from \cite[Tables 8.2, 8.3, 8.8, and 8.18]{BHRD} that $|T:M|$ has the form $\frac{q^n-1}{(q-1)}$ or $\frac{(q^n-1)(q^{n-1}-1)}{(q-1)(q^2-1)}$. A routine exercise then shows that only the former can take the form $2^{x_1}3^{y_1}5^{z_1}$ with $0\le y_1,z_1\le 1$, and this can only happen if $n\in\{2,4\}$ and $q$ has the form $q=2^a3^b5^c-1$. It is then easy to see that the only possibilities are either $(n,q)=(4,3)$ or $(n,q)=(2,2^{x_1}3^{y_1}5^{z_1}-1)$ with $0\le y_1,z_1\le 1$. If $(n,q)=(4,3)$, then the only possibility is $M=H=3^3:\Lg_3(3)$. If $(n,q)=(2,2^{x_1}3^{y_1}5^{z_1}-1)$ then $2^{x_1}$ is the largest power of $2$ dividing $|T|$, so $x=x_1$ and $H$ has index dividing $3^{y-y_1}5^{z-z_1}$ in $M$. 
So we may assume that $M$ is not parabolic. Then $p$ divides $|T:M|$, so $p\in\{2,3,5\}$. Also, as mentioned above, $r=5$ does not divide $|M|$. We can then quickly see from \cite[Tables 8.2, 8.3, 8.8, and 8.18]{BHRD} that the only possibility is $(n,q)=(3,2)$ with $H$ having order $7$ or $21$. Since $\Lg_3(2)\cong \Lg_2(7)$, this case is already listed in (ii)(a).
This completes the proof in the case $T=\Lg_n(q)$.

The other classical cases are entirely similar: Suppose that $T=\Ug_n(q)$, with $3\le n\le 6$. If $M$ is parabolic, then $|T:M|$ is divisible by either $\frac{(q^n-(-1)^n)(q^{n-1}-(-1)^{n-1})}{q^2-1}$ or $\prod_{i=1}^{\left\lfloor\frac{n}{2}\right\rfloor}(q^{2i+1}+1)$. It is easy to see, however, that neither $q^3+1$ not $q^5+1$ can ever have the form $2^{x_1}3^{y_1}5^{z_1}$ with $0\le y_1,z_1\le 1$. So $M$ must be non-parabolic. Then $p$ divides $|T:M|$, so $p\in\{2,3,5\}$. Also, $5$ does not divide $|M|$. We can then use the tables in \cite[Chapter 8]{BHRD} to quickly find that the only case of a maximal non-parabolic subgroups $M$ of $T$ with $|T:M|=2^{x_1}3^{y_1}5^{z_1}$ and $0\le y_1,z_1\le 1$ occurs when $(n,q)=(4,2)$. Since $\Ug_4(2)\cong \PSp_4(3)$, this case is accounted for in (iii).
This completes the proof in the case $T=\Ug_n(q)$.     

The arguments in the remaining classical cases are almost identical, so we now move on to the case where $G$ is an exceptional group of Lie type. Then $r=5$ is one of the primes occurring in the second column of the row for $T$ in \cite[Table 10.5]{LPS}. By using (\ref{lab:ppd}), we see that the only possibilities are $T\in\{{}^2F_4(q)',{}^2B_2(q)\}$, and $5$ is a primitive prime divisor of $q^4-1$. In each of these cases, the maximal subgroups of $T$ are known: see \cite{Malle2F4} for $T={}^2F_4(q)'$ and \cite{SuzB2} for $T={}^2B_2(q)$. In each case, we can quickly check that no maximal subgroup $M$ of $T$ with $\pi(|T|)\neq \pi(|M|)$ can have $|T:M|$ of the form $2^{x_1}3^{y_1}5^{z_1}$, for $0\le y_1,z_1\le 1$.

Finally, assume that $G$ is a sporadic simple group. Then $T=J_2$, by \cite[Table 10.6]{LPS}. However, a quick check of the Web Atlas \cite{WebAtlas} shows that no maximal subgroup $M$ of $J_2$ with $\pi(|T|)\neq \pi(|M|)$ can have $|T:M|$ of the form $2^x3^y5^z$, for $0\le y,z\le 1$. This completes the proof.
\end{proof}

We are now ready to determine the possibilities for the pairs ($T$,$H$), with $T$ simple, $|T:H|$ dividing $2^{x}15$, and $H$ contained in a maximal subgroup $M$ of $T$ with $\pi(|T|)=\pi(|M|)$. 
\begin{prop}\label{prop:Pi} Let $T$ be a nonabelian finite simple group, and suppose that $T$ has a subgroup $H$ of index $n=2^x3^y5^z$, with $0\le y,z\le 1$. Suppose also that $T$ is not an alternating group of degree less than $10$. Let $M$ be a maximal subgroup of $T$ containing $H$, and assume that $\pi(|T|)=\pi(|M|)$. Then one of the following holds:\begin{enumerate}[(i)]
\item $T=\Alt_n$, and either $H=\Alt_{n-1}$, or $n=16$ and $\Alt_{14}\le H\le \Alt_{14}.2$.
\item $T=\PSp_{2m}(2^f)$ with $(m,f)\in\{(4,1),(2,2)\}$, and $\POmega^{-}_{2m}(2^f) \unlhd H$.
\item $T=\PSp_{4}(q)$, with $q\in\{2,4\}$, and $\PSp_{2}(q^{2}) \unlhd H$.
\item $T=M_{11}$ and $H=\Lg_{2}(11)$;
\item $T=\POmega^{+}_{8}(2)$ and $H=\Alt_9$ (three classes, fused by a triality automorphsim); or
\item $T=M_{12}$ or $M_{24}$, and $H=M_{11}$ or $H=M_{23}$, respectively.
\end{enumerate}\end{prop}
\begin{proof} We can quickly check using direct computation that $T$ can have no maximal subgroup $M$ of the required index when $T$ is $\Lg_{2}(8)$, $\Lg_{3}(3)$, $\Ug_{3}(3)$ or $\PSp_{4}(8)$. So we may assume that $T$ is not one of these groups. Suppose first that $T\neq \Lg_{2}(p)$ for a Mersenne prime $p$. Then using \cite[Corollary 6 and Table 10.7]{LPS}, the possibilities for $T$ and $M$ are as follows.\begin{enumerate}[(1)]
\item $T=\Alt_{n}$ and $M=\Alt_{k}\cap (\Sym_{k}\times \Sym_{n-k})$ for some $k\le n-1$ with $k\geq p$ for all primes $p\le n$. Then $|\Alt_{n}:\Alt_{n}\cap (\Sym_{k}\times \Sym_{n-k})|=\binom{n}{k}$ divides $|T:H|=2^x3^y5^z$. By a well-known theorem of Sylvester and Schur (see \cite{Syl}), $\binom{n}{k}$ has a prime divisor exceeding $\min\left\{k,n-k\right\}$. Thus $k\in\{n-1,n-2,n-3,n-4\}$, since $k\geq 10$. We then quickly see that the only possibilities are $H=\Alt_{n-1}$; or $n=16$ and $k=14$, which gives us what we need.
\item $T=\PSp_{2m}(q)$ ($m$, $q$ even) or $\POmega_{2m+1}(q)$ ($m$ even, $q$ odd), and $M=N_T(\POmega^{-}_{2m}(q))$. Now, $|N_{T}(\POmega^{-}_{2m}(q)):\POmega^{-}_{2m}(q)|\le 2$ using \cite[Corollary 2.10.4(i) and Table 2.1.D]{KL}. Hence, $|T:\POmega^{-}_{2m}(q)|$ divides $2^{x+1}3^y5^{z}$. Also, for each of the two choices of $T$ we get $|T:\POmega^{-}_{2m}(q)|=q^{m}(q^{m}-1)$. But $q^{m}(q^{m}-1)$ has the form $2^{x_1}3^{y_1}5^{z_1}$ for $0\le y_1,z_1\le 1$ if and only if $q=2^f$ with $fm\in\{2,4\}$, since $m$ is even. Since $\PSp_4(2)\cong\Sym_6$ is not simple, we deduce that ($m,f$)$\in\{(4,1),(2,2)\}$.
 \item $T=\POmega^{+}_{2m}(q)$ ($m$ even, $q$ odd) and $M=N_T(\POmega_{2m-1}(q))$. By \cite[Corollary 2.10.4 Part (i) and Table 2.1.D]{KL}, we have $|N_{T}(\POmega_{2m-1}(q)):\POmega_{2m-1}(q)|\le 2$. It follows that $\frac{1}{2}q^{m-1}(q^{m}-1)=|T:\POmega_{2m-1}(q)|$ divides $2^{x+1}3^y5^{z}$. However, since $q$ is odd, $m\geq 4$, and $0\le y,z\le 1$, this cannot be the case.
\item $T=\PSp_{4}(q)$ and $M=N_T(\PSp_{2}(q^{2}))$. Then \cite[Corollary 2.10.4 Part (i) and Table 2.1.D]{KL} gives $|N_{T}(\PSp_{2}(q^{2})):\PSp_{2}(q^{2})|\le 2$. It follows that $q^{2}(q^{2}-1)=|T:\PSp_{2}(q^{2})|$ divides $2^{x+1}3^y5^{z}$. Thus, $q\in\{2,4\}$. We then use direct computation to see that the only possibility is $\PSp_{2}(q^{2}) \unlhd H$.
\item In each of the remaining cases (see \cite[Table 10.7]{LPS}), we are given a pair ($T$, $Y$), where $T$ is $\Lg_{2}(8)$, $\Lg_{3}(3)$, $\Lg_{6}(2)$, $\Ug_{3}(3)$, $\Ug_{3}(3)$, $\Ug_{3}(5)$, $\Ug_{4}(2)$, $\Ug_{4}(3)$, $\Ug_{5}(2)$, $\Ug_{6}(2)$, $\PSp_{4}(7)$, $\PSp_{4}(8)$, $\PSp_{6}(2)$, $\POmega^{+}_{8}(2)$, $G_{2}(3)$, $^{2}F_{4}(2)'$, $M_{11}$, $M_{12}$, $M_{24}$, $HS$, $McL$, $Co_{2}$ or $Co_{3}$, and $Y$ is a subgroup of $T$ containing $H$. Apart from when $T=\POmega^{+}_{8}(2)$, $M_{11}$, $M_{12}$, and $M_{24}$, we find that $|T:Y|$ does not divide $2^x3^y5^z$, so we get a contradiction in all other cases. When $T=\POmega^{+}_{8}(2)$, the only possibilities are $H=\Alt_9$ of index $2^615$ (three $T$-conjugacy classes fused under the triality automorphism). 
When $T=M_{11}$, $M_{12}$, or $M_{24}$, the only possibilities for $H$ are $H=\Lg_{2}(11)\le M_{11}$ (of index $12$), $H=M_{11}\le M_{12}$ (of index $12$), or $H=M_{23}\le M_{24}$ (of index $24$).\end{enumerate}
Finally, assume that $T=\Lg_{2}(p)$, for some Mersenne prime $p$, and let $M$ be a maximal subgroup of $T$ containing $H$. Then, since $|T:M|$ divides $|T:H|=2^x3^y5^z$ with $0\le y,z\le 1$, $M$ must be parabolic (see \cite[Table 8.1]{BHRD}). But then $\pi(|T|)\neq \pi(|M|)$, since $p+1$ is the highest power of $2$ dividing $|T|$ -- a contradiction. This completes the proof.\end{proof}

This completes our analysis of the pairs $(T,H)$ with $T$ a nonabelian simple group and $H$ a subgroup of $T$ of index $|T:H|=2^x3^y5^z$, with $0\le y,z\le 1$. The possibilities are listed in Proposition \ref{prop:nonPi}(i)--(iii) and Proposition \ref{prop:Pi}(i)--(vi).

In particular, if $G$ is a counterexample to Theorem \ref{thm:MinTransTheorem} of minimal degree, $L=T^e$ is a minimal normal subgroup of $G$, and $H$ is a point stabilizer in $L$, then we know the possibilities for the group $T$, and the intersections of $H$ with the $i$th coordinate subgroups in $L$.
Our next task is to use this information to find out more about the structure of $H$.

One of our main tools for doing this is the \emph{Frattini argument}. The proof is well-known, but we couldn't find a reference so we include a proof here.
\begin{lem}\label{FrattiniArgument} Let $G$ be a finite group, and let $L$ be a normal subgroup of $G$. Let $A\le \Aut(L)$ be the image of the induced action of $G$ on $L$ (by conjugation). Suppose that $H$ is a subgroup of $L$ with the property that $H$ and $H^{\alpha}$ are $L$-conjugate for each $\alpha\in A$. Then $G=N_G(H)L$.\end{lem}
\begin{proof} Let $g\in G$. Then $H^g=H^l$ for some $l\in L$, by hypothesis. Hence, $gl^{-1}\in N_{G}(H)$, so $g\in N_{G}(H)L$, and this completes the proof.\end{proof}

With the Frattini argument in mind, the next lemma will be crucial.
\begin{lem}\label{lem:FratStep} Let $T$ be a nonabelian finite simple group, and suppose that $T$ has a subgroup $H$ of index $n:=2^x3^y5^z$, with $0\le y,z\le 1$. Assume also that:
\begin{itemize}
    \item If $T=\Alt_5$, then $H$ is either transitive in its natural action, or $H=\Alt_4$;
    \item if $T=\Alt_6$, then $H$ is either transitive in its natural action, or $H=\Alt_5$ is a natural point stabilizer in $T$;
    \item if $T= \Lg_{2}(q)$, with $q$ of the form $q=2^{x_1}3^{y_1}5^{z_1}-1$, then $y=y_1$ and $z=z_1$;
    \item if $T=\PSp_{4}(4)$, then $H\neq \PSp_2(16)$ and $H\neq\PSp_2(16).2$;
    \item if $T=\POmega^{+}_{8}(2)$, then $H\neq\Alt_9$ (three classes);
    \item if $T=\Alt_{16}$, then $H\neq \Alt_{14}$ and $H\neq \Alt_{14}.2$; and
    \item $(T,H)$ is not $(\Alt_8,H)$, with $H$ of shape $7:3$, $2^3:(7:3)$, $\GL_3(2)$, or $2^3:\GL_3(2)$.
\end{itemize}
 Denote by $\Sigma$ the set of right cosets of $H$ in $T$. 
Then there exists a proper subgroup $S_0$ of $T$ with the following properties:\begin{enumerate}[(i)]
\item $S_0$ and $S_0^{\alpha}$ are conjugate in $T$ for each $\alpha\in \Aut(T)$;
\item $N_T(S_0)^{\Sigma}$ is transitive;
\item $N_T(S_0)$ acts transitively on the set of cosets of any subgroup of $T$ of $2$-power index.
\end{enumerate}\end{lem}
\begin{proof} Note first that if $T=\Alt_n$ and $K$ is a subgroup of $T$ of $2$-power index, then $K=\Alt_{n-1}$ and $n$ is a power of $2$, by Proposition \ref{prop:nonPi} and \ref{prop:Pi}. Thus, if $T=\Alt_n$, with $n$ not a power of $2$, and $H$ is transitive, then setting $S_0$ to be a natural point stabilizer $S_0:=\Alt_{n-1}$ suffices.
Also, if $|T:H|$ is a power of $2$, then setting $S_0$ to be a Sylow $2$-subgroup of $T$ gives us what we need. So we may assume that we are in neither of these cases.

By Proposition \ref{prop:nonPi} and \ref{prop:Pi}, the remaining possibilities for the pair $(T,H)$ are as follows:\begin{enumerate}[(1)]
\item $(T,H)=(\Alt_{n},\Alt_{n-1})$, with $n=2^x3^y5^z$ and $y+z\neq 0$. 
Set $S_0=\langle (1,\hdots,3^y5^z), (3^y5^z+1,\hdots,2\times 3^y5^z),\hdots,(n-3^y5^z+1,\hdots,n)\rangle$. Then $N_{T}(S_0)^{\Sigma}$ is transitive, and $S_0$ and $S_0^{\alpha}$ are $T$-conjugate for all $\alpha\in\Aut(T)$ (this includes the case $n=6$, when $\Out{(\Alt_{6})}$ has order $4$). Thus, (i) and (ii) are satisfied. Property (iii) is vacuously satisfied since $T$ has no proper subgroup of $2$-power index in this case, as mentioned at the beginning of the proof.
\item $T=\PSp_{8}(2)$, and $\POmega^{-}_{8}(2) \unlhd H$. Note that property (iii) is vacuously satisfied in each case, since $T$ has no proper subgroup of $2$-power index. Moreover, $T$ has a maximal parabolic subgroup $S_0$ of shape $2^{10}.GL_4(2)$ satisfying (i) and (ii).

\item $T= \Lg_4(3)$, $H=3^3:\Lg_3(3)$ (two classes). In this case, take $S_0$ to be the unique maximal subgroup of $T$ of shape $\Lg_2(9).4.2$. Then it is easily seen that (i) and (ii) hold, for either of the two non-conjugate choices for $H$. Property (iii) holds since $T$ has no proper subgroups of $2$-power index.

\item $T= \PSp_4(3)$, and $|T:H|=40$ (two classes). Here, the unique maximal subgroup $S_0$ of $T$ of index $27$ has properties (i) and (ii) for each of two non-conjugate choices for $H$. Property (iii) again holds since $T$ has no proper subgroups of $2$-power index.

\item $T= \PSp_6(2)$, and $|T:H|\in\{120,960\}$. Taking $S_0$ to be the unique maximal subgroup of $T$ of index $36$ gives us (i) and (ii) in either of the cases $|T:H|=120$ and $|T:H|=960$. Property (iii) holds since $T$ has no proper subgroups of $2$-power index.
\item $(T,H)=(M_{12},M_{11})$ or $(M_{24},M_{23})$: In each case, let $S_0$ be a subgroup of $T$ generated by a fixed point free element of order $3$. When $T=M_{12}$, $N_{T}(S_0)\cong \Alt_{4}\times \Sym_{3}$ (see \cite{WebAtlas}) is a maximal subgroup of $T$, and acts transitively on the cosets of $H$. Thus, (ii) holds. Also, outer automorphisms of $M_{12}$ fix the set of $T$-conjugates of $S_0$, so $S_0$ and $S_0^{\alpha}$ are $T$-conjugate for all $\alpha\in\Aut(T)$. Hence, (i) holds.
When $T=M_{24}$, $N_{T}(S_0)$ has order $1008$, and acts transitively on the cosets of $H$ (using \Magma \cite{MR1484478}, for example). Also, $\Out{(T)}$ is trivial. Thus, (i) and (ii) are again satisfied. Property (iii) is vacuously satisfied in each case, since $T$ has no proper subgroup of $2$-power index.
\item $T=\Lg_{2}(q)$, with $q$ of the form $q=2^{x_1}3^{y_1}5^{z_1}-1$, $0\le y_1,z_1\le 1$. Since we are assuming that $y=y_1$ and $z=z_1$, $H$ must be a maximal parabolic subgroup of $T$ (see \cite[Table 8.1]{BHRD}).
Let $S_0$ be a $D_{q+1}$ maximal subgroup of $T$. Then $S_0$ and $S_0^{\alpha}$ are $T$-conjugate for all $\alpha\in\Aut(T)$, so (i) holds. Clearly, (ii) also holds. The group $T$ only has a proper subgroup of $2$-power index if $y_1=z_1=0$, and in this case the only such subgroup is $H$. Thus, property (iii) is also satisfied.
\end{enumerate}\end{proof}

\begin{cor}\label{Cor:FratNew}
Let $G$ be a minimal transitive group of degree $n=2^x3^y5^z$, where $0\le y,z\le 1$. Suppose that $G$ has a nonabelian minimal normal subgroup $L=T^e$, and let $H$ be the intersection of $L$ with a point stabilizer in $G$. Write $|L:H|=2^{x_0}3^{y_0}5^{z_0}$, where $x_0\le x$, $y_0\le y$, $z_0\le z$. Then:
\begin{enumerate}[(i)]
    \item Either $y_0$ or $z_0$ is non-zero;
    \item Suppose that $H=H\pi_1\times\hdots\times H\pi_e$, and all but one of the groups $H\pi_i$, say $H\pi_e$, has $2$-power index in $T$. Then ($T$,$H\pi_e$) cannot satisfy the hypothesis of Lemma \ref{lem:FratStep}.  
\end{enumerate}
\end{cor}
\begin{proof}
Suppose that either $|L:H|$ is a power of $2$; or $H=H\pi_1\times\hdots\times H\pi_e$ where all but one of the groups $H\pi_i$, say $H\pi_e$, has $2$-power index in $T$, and ($T$,$H\pi_e$) satisfies the hypothesis of Lemma \ref{lem:FratStep}. Then we claim that there exists a proper subgroup $S$ of $L$ with $S$ acting transitively on the cosets of $H$ in $L$, and $LN_G(S)=G$ ($\ast$). In particular, $N_G(S)$ is a proper subgroup of $G$, since $L$ is a minimal normal subgroup of $G$. But then $N_G(S)<G$ acts transitively both on an $L$-orbit, and on the set of $L$-orbits, so $N_G(S)$ is transitive. This contradicts the minimal transitivity of $G$. Thus, it will suffice to prove the existence of such a subgroup $S$. 

Now, if $|L:H|$ is a power of $2$, then $S\in\Syl_2(L)$ satisfies ($\ast$), by Lemma \ref{FrattiniArgument} and the fact that $(|L:H|,|L:S|)=1$. 
So assume that $H=H\pi_1\times\hdots\times H\pi_e$, and all but one of the groups $H\pi_i$, say $H\pi_e$, has $2$-power index in $T$. Assume that the pair ($T$,$H\pi_e$) satisfies the hypothesis of Lemma \ref{lem:FratStep}, and let $S_0$ be the subgroup of $T$ exhibited therein. Set $S:=S_0^e<L$.
Then $S$ acts transitively on the cosets of $H$ in $L$. Moreover, since $\Aut(L)\cong \Aut(T)\wr\mathrm{Sym}_e$,
Lemma \ref{lem:FratStep}(i) implies that $S$ and $S^{\alpha}$ are $L$-conjugate for all $\alpha\in \Aut(L)$. Thus, $LN_G(S)=G$ by Lemma \ref{FrattiniArgument}, so $S$ satisfies ($\ast$), and this completes the proof.
\end{proof}


In what follows, recall from Definition \ref{Def:Subd} that \emph{a subdirect product of the form} $G=\frac{1}{n}(G_1\times\hdots\times G_e)$ is a subdirect product of index $n$ in $G_1\times \hdots\times G_e$ containing $[G_1,G_1]\times\hdots\times [G_e,G_e]$.
\begin{cor}\label{cor:LPoss}
Let $G$ be a minimal transitive group of degree $n=2^x3^y5^z$, where $0\le y,z\le 1$. Suppose that $G$ has a nonabelian minimal normal subgroup $L=T^e$, and write $\pi_i:L\rightarrow T$ for the $i$th coordinate projections. Let $H$ be the intersection of $L$ with a point stabilizer in $G$. Write $|L:H|=2^{x_0}3^{y_0}5^{z_0}$, where $x_0\le x$, $y_0\le y$, $z_0\le z$. Then one of the following holds: 
\begin{enumerate}[(1)]
\item $T=\Alt_n$ with $n\in\{5,6\}$ and one of the following holds:
\begin{enumerate}
    \item $H$ is $\mathrm{Sym}_e$-conjugate to $T^{e-2}\times H_{e-1}\times H_e$, where $H_{e-1}$ and $H_e$ are subgroups of $T$ with $|T:H_{e-1}|=2^{x_{e-1}}3$, $|T:H_{e}|=2^{x_e}5$, and $x_0=x_{e-1}+x_e$;
    \item $H$ is $\mathrm{Sym}_e$-conjugate to $T^{e-1}\times H_e$, where $H_e\le T$ is not transitive in its natural action; $H_e$ is not $T$-conjugate to a natural point stabilizer $\Alt_{n-1}$; and $|T:H_e|=|L:H|$;
    \item $T=\Alt_5$ and $H$ is $\mathrm{Sym}_e$-conjugate to $T^{e-2}\times J$, where $J<T^2$ is a subdirect product of the form $J=\frac{1}{2}(\Sym_3\times D_{10})< T^2$; or
    \item $T=\Alt_5$ and $H$ is $\mathrm{Sym}_e$-conjugate to $T^{e-2}\times J$, where $J$ is a diagonal subgroup of $T^2$.
\end{enumerate}
\item $T=\Alt_8$ and $H$ is $\mathrm{Sym}_e$-conjugate to $T^{e_1}\times H_{e_1+1}\times\hdots\times H_{e-1}\times H_e$, where $H_e$ is one of the subgroups of shape $7:3$, $2^3:(7:3)$, $\GL_3(2)$, or $2^3:\GL_3(2)$ in $T$; $H_i$ is $T$-conjugate to $\Alt_7$ for all $e_1+1\le i\le e-1$; and $0\le e_1\le e-1$.
\item $T=\Alt_{16}$ and $H$ is $\mathrm{Sym}_e$-conjugate to $T^{e_1}\times H_{e_1+1}\times\hdots\times H_{e-1}\times H_e$, where $H_e$ is one of the subgroups of shape $\Alt_{14}$ or $\Alt_{14}.2$ in $T$; $H_i$ is $T$-conjugate to $\Alt_{15}$ for all $e_1+1\le i\le e-1$; and $0\le e_1\le e-1$.
\item $T=\Lg_2(q)$, with $q$ of the form $q=2^{x'}3^{y'}5^{z'}-1$, $0\le y',z'\le 1$, and one of the following holds:
\begin{enumerate}
    \item $y'=z'=0$; $5\mid q-1$; $H$ is $\mathrm{Sym}_e$-conjugate to $T^{e_1}\times P_1\times\hdots\times P_{e_2}\times H_{e-1}\times H_e$, where each $P_i$ is a maximal parabolic subgroup of $T$; $H_{e-1}$ [respectively $H_e$] is a subgroup of index $3$ [resp. $5$] in a maximal parabolic subgroup of $T$; and $e_1+e_2=e-2$;
    \item $y'=z'=0$ and $H$ is $\mathrm{Sym}_e$-conjugate to $T^{e_1}\times L_1$, where $L_1$ is a subdirect product of the form $L_1=\frac{1}{t}(L_2\times H\pi_e)$; $L_2$ is a subdirect product of the form $L_2=\frac{1}{s}(P_1\times\hdots\times P_{e_2})$; each $P_i$ is a maximal parabolic subgroup of $T$; $H\pi_e$ is a subgroup of index $r$ in a maximal parabolic subgroup of $T$; $rst=3^{y_0}5^{z_0}$; and $e_1+e_2=e-1$; or
    \item $y'+z'=1$, and $H$ is $\mathrm{Sym}_e$-conjugate to $T^{e-1}\times \left(q:\frac{q-1}{2r}\right)$, where $r>1$ divides $q-1$, and $3^{y'}5^{z'}r=3^{y_0}5^{z_0}$.
\end{enumerate}
\item $T=\PSp_4(4)$ and $H$ is $\mathrm{Sym}_e$-conjugate to $T^{e-1}\times H_e$, where $H_e\in\{\PSp_2(16)\text{ }\mathrm{ (}\text{two }T\text{-classes}\mathrm{) }, \\ \PSp_2(16).2\text{ }\mathrm{ (}\text{two }T\text{-classes}\mathrm{)}\}$.
\item $T=\POmega^+_8(2)$ and $H$ is $\mathrm{Sym}_e$-conjugate to $T^{e-1}\times H_e$, where $H_e=\Alt_9$ (three $T$-classes, fused under a triality automorphism).
\end{enumerate}
\end{cor}
\begin{proof}
By Corollary \ref{Cor:FratNew}, either $y_0$ or $z_0$ is non-zero. Now, set $H_i:=(T_i\cap H)\pi_i$, for $1\le i\le e$. Then $H_i\unlhd H\pi_i$, and $|T:H_i|$ has the form $2^{a_i}3^{b_i}5^{c_i}$ for each $i$.
Note also that since $H$ is a subgroup of $H\pi_1\times\hdots\times H\pi_e\le L$, we have that $\prod_{i=1}^e|T:H\pi_i|$ divides $|L:H|$. Thus, $|T:H\pi_i|$ has the form $2^{x_i}3^{y_i}5^{z_i}$ for each $i$, and $y_i$ [respectively $z_i$] is $1$ for at most $y_0$ [resp. $z_0$] values of $i$. It follows that $|T:H\pi_i|$ is a power of $2$ for all but at most two values of $i$. Let $\Lambda:=\{i\text{ : }|T:H\pi_i|\text{ is a power of }2\}$, so that $|\Lambda|\geq t-2$. For a subset $E$ of $\{1,\hdots,e\}$, let $\pi_{E}:L\rightarrow \prod_{i\in E}T_i$ be the natural projection homomorphism onto the `$E$-part' of $L$. Then since $H$ is a subgroup of $H\pi_{\Lambda}\times \prod_{i\not\in\Lambda}H\pi_{i}$ (up to permutation of coordinates), we have that 
\begin{align}\label{Lamb}
 \text{$|T^{|\Lambda|}:H\pi_{\Lambda}|\prod_{i\not\in\Lambda}|T:H\pi_i|$ divides $|L:H|$.}   
\end{align}
Furthermore, we can determine the subgroups of $2$-power index in the simple groups $T$ from Propositions \ref{prop:nonPi} and \ref{prop:Pi}: we see that either\begin{enumerate}[\upshape(I)]
\item $H\pi_i=T$ for all $i\in\Lambda$
\item $T=\Alt_{2^{x_i}}$ and $H\pi_i$ is $T$-conjugate to either $T$ or $\Alt_{2^{x_i}-1}$ for all $i\in\Lambda$; or 
\item $T=\Lg_2(p)$, with $p=2^{x_i}-1$, and $H\pi_i$ is either $T$ or a maximal parabolic subgroup of $T$, for all $i\in\Lambda$.
\end{enumerate}

Assume first that $|\Lambda|=t-2$, and without loss of generality, write $[i]\setminus\Lambda=\{e-1,e\}$. Then we may assume, again without of loss of generality, that $|T:H\pi_{e-1}|=2^{x_{e-1}}3$ and $|T:H\pi_{e}|=2^{x_e}5$. It then follows from (\ref{Lamb}) that
\begin{align}\label{Lamb0}
|T^{e-2}:H\pi_{\Lambda}|\text{ is a power of }2. 
\end{align}
Next, from Propositions \ref{prop:nonPi} and \ref{prop:Pi}, we can determine all finite simple groups $T$ which contain both a subgroup of index $2^{x_{e-1}}3$ and a subgroup of index $2^{x_e}5$: we see that $T\in\{\Lg_2(p),\Alt_5$, $\Alt_6\text{ : }p\text{ a Mersenne prime, } 15\mid p-1\}$. Suppose first that $T=\Alt_n$, for $n\in\{5,6\}$. Then since $T$ is the only subgroup of $T$ of $2$-power index, we have $H\pi_i=T$ for all $i\in \Lambda$. Thus, 
$H$ is $\mathrm{Sym}_e$-conjugate to $T^{e-2}\times J$, where $J$ is a subdirect subgroup of $H\pi_{e-1}\times H\pi_e$, by Lemma \ref{lem:Subd} and (\ref{Lamb0}). We can now quickly determine, using \Magma \cite{MR1484478} for example, all subgroups $J$ of $T^2$ where $J\pi_{e-1}$, $J\pi_e<T$, and $|T^2:J|$ has the form $2^A3^B5^C$. This yields cases (1)(a) and (1)(c) in the statement of the corollary.

Suppose next that $T=\Lg_2(p)$, where $p$ is a Mersenne prime. Then the only subgroups of $T$ of $2$-power index are $T$ itself, together with the maximal parabolic subgroup $P<T$, of index $p+1$. It follows from (\ref{Lamb0}) and Lemma \ref{lem:Subd}(v) that 
$H\pi_{\Lambda}=T^{e_1}\times L_2$ (up to permutation of coordinates), where $L_2$ is a subdirect product in $P_1\times\hdots\times P_{e_2}$, each $P_i$ is $T$-conjugate to $P$, and $e_1+e_2=e$. Also, since $P$ has odd order, (\ref{Lamb0}) implies that
that $L_2=P_1\times\hdots\times P_{e_2}$. Hence, $H\pi_{\lambda}=T^{e_1}\times P_1\times\hdots\times P_{e_2}$ (up to permutation of coordinates). Thus, $H$ is a subdirect product in $T^{e_1}\times P_1\times\hdots\times P_{e_2}\times J$, where $J:=H\pi_{\{e-1,e\}}\le H{\pi_{e-1}}\times H\pi_e$. Now, the only subgroup of $T$ (up to $T$-conjugacy) with index $2^{x_{e-1}}3$ is the unique subgroup $Y_3$ of $P$ with $|P:Y|=3$ , $|T:Y|=3(p+1)$; while the only subgroup of $T$ (up to $T$-conjugacy) with index $2^{x_{e-1}}5$ is the unique subgroup $Y_5$ of $P$ with $|P:Y|=5$ , $|T:Y|=5(p+1)$. Thus, $J$ is a subdirect product in $H\pi_{e-1}\times H{\pi_e}\cong p:\frac{p-1}{6}\times p:\frac{p-1}{10}$. Since $|T^2:J|$ has the form $2^A3^b5^C$ and $|J|$ is odd, we must have $J=H\pi_{e-1}\times H\pi_{e-1}\cong p:\frac{p-1}{6}\times p:\frac{p-1}{10}$. Moreover, since $|T_{e-1}\times T_e:H\cap (T_{e-1}\times T_e)|$ has the form $2^A3^B5^C$, we must have $J=H\pi_{e-1}\times H\pi_e=H\cap (T_{e-1}\times T_e)$. It follows that $H$ is $\mathrm{Sym}_e$-conjugate to $T^{e_1}\times P_1\times\hdots\times P_{e_2}\times H_{e-1}\times H_e$, and this yields case (4)(a) in the statement of the corollary.

Suppose next that $|\Lambda|=t-1$. Without loss of generality, we may assume that $[i]\setminus\Lambda=\{e\}$. Hence, $|T:H\pi_e|=2^{x_{e}}r$, with $r\in\{3,5,15\}$.
It then follows from (\ref{Lamb}) that 
\begin{align}\label{Lamb2}
\text{$|T^{e-1}:H\pi_{\Lambda}|2^{x_e}r$ divides $|L:H|$ for some $r\in\{3,5,15\}$.}    
\end{align}
In particular, $|T|$ does not divide $|T^{e-1}:H\pi_{\Lambda}|$.
Suppose that case (I) occurs (recall that (I) is defined at the end of the first paragraph of this proof). Then $H$ is $\mathrm{Sym}_e$-conjugate to $T^{e-1}\times H_e$ by (\ref{Lamb2}) and Lemma \ref{lem:Subd}(v).   
Suppose now that we are in case (II). Then $H\pi_{\Lambda}$ is a subdirect product in $T^{e_1}\times H_{e_1+1}\times\hdots\times H_{e-1}$, where $0\le e_1\le e-1$, and each $H_i$ is $T$-conjugate to $\Alt_{2^{x_i}-1}$. The analogous argument to the above then shows that $H$ is $\mathrm{Sym}_e$-conjugate to $T^{e_1}\times H_{e_1+1}\times\hdots\times H_{e}$, where $0\le e_1\le e-1$. By Corollary, \ref{Cor:FratNew}, the pair ($T$,$H_e$) cannot satisfy the hypothesis of Lemma \ref{lem:FratStep}, in either of the cases (I), (II). Since $|T:H_e|$ has the form $2^A3$, $2^A5$, or $2^A15$, Propositions \ref{prop:nonPi} and \ref{prop:Pi} then imply that one of the following holds:
\begin{itemize}
    \item $T=\Alt_n$ with $n\in\{5,6\}$, and $H_e$ is both intransitive and not conjugate to  $\Alt_4$;
    \item $T=\Alt_8$, and $H_e$ has shape $7:3$, $2^3:(7:3)$, $\GL_3(2)$, or $2^3:\GL_3(2)$;
    \item $T=\Alt_{16}$, and $H_e\in\{\Alt_{14},\Alt_{14}.2\}$; 
    \item $T=\PSp_{4}(4)$, and either $H_e\cong \PSp_2(16)$ or $H_e\cong\PSp_2(16).2$;
    \item $T=\POmega^{+}_{8}(2)$, and $H\cong\Alt_9$;
    or
    \item $T=\Lg_2(q)$ with $q=2^{x'}3^{y'}5^{z'}-1$; $|L:H|=r(q+1)$ for some $r\in\{3,5\}$ (so either $y'$ or $z'$ is $0$), and $H_e\cong q:\frac{q-1}{2r}$.
\end{itemize}
The first five of these are cases (1)(b), (2), (3), (5), and (6), respectively, in the statement of the corollary. The last is case (4)(b) (with $s=t=1$) if $y'=z'=0$, and case (4)(c) otherwise.
Finally, assume that case (III) holds. Then using (\ref{Lamb2}) and Lemma \ref{lem:Subd} again, we see that $H$ is $\mathrm{Sym}_e$-conjugate to $T^{e_1}\times L_1$, where $L_1$ is a subdirect product in $L_2\times H\pi_e$; $L_2$ is a subdirect product in $P_1\times\hdots\times P_{e_2}$; each $P_i$ is a maximal parabolic subgroup of $T$; and $e_1+e_2=e-1$. Moreover, $H\pi_{\Lambda}=T^{e_1}\times L_2$ (up to permutation of coordinates). In particular, $|H|=|T|^{e_1}|L_2||H_e|$. Thus, $|L:H|=|T^{e_2}:L_2||T:H\pi_e||H\pi_e:H_e|$. Let $s$ and $t$ be the odd parts of $|T^{e_2}:L_2|$ and $|H\pi_e:H_e|$, respectively. Then $s$ is the index of $L_2$ in $P_1\times\hdots\times P_{e_2}$, $rst=3^{y_0}5^{z_0}$, and
one of the following must hold:
\begin{itemize}
\item $s=t=1$ and $H$ is $\mathrm{Sym}_e$-conjugate to $T^{e_1}\times L_2\times H\pi_e$, $L_2=P_1\times\hdots\times P_{e_2}$.
\item $s\in\{3,5\}$, $t=1$, and $H$ is $\mathrm{Sym}_e$-conjugate to $T^{e_1}\times L_2\times H\pi_e$, where $L_2$ has the form $L_2=\frac{1}{s}(P_1\times\hdots\times P_{e_2})$.
\item $t\in\{3,5\}$, $s=1$, and $H$ is $\mathrm{Sym}_e$-conjugate to $T^{e_1}\times L_1$, where $L_1$ has the form $L_1=\frac{1}{t}(P_1\times\hdots\times P_{e_2}\times p:\frac{p-1}{2r})$.
\end{itemize}
These are the cases in (4)(b) in the statement of the corollary.

Finally, assume that $|\Lambda|=t$. Suppose first that case (I) holds (see the first paragraph above), so that $H$ is a subdirect product in $L=T^e$. Then $|T|$ divides $|L:H|$ by Lemma \ref{lem:Subd}(iii), since $H\neq L$. Thus, we must have $T=\Alt_5$, since $|L:H|$ divides $2^{x_0}15$. Moreover, $H$ is $\mathrm{Sym}_e$-conjugate to $T^{e-2}\times J$, where $J$ is a diagonal subgroup of $T^2$, since $|T|^2$ cannot divide $|L:H|$. This is case (1)(d) in the statement of the corollary. Suppose next that case (II) holds. Then since neither $|\Alt_{2^{x_i}}|$ nor $|\Alt_{2^{x_i}-1}|$ can divide $|L:H|$ in this, case another application of Lemma \ref{lem:Subd} yields that $H$ is $\mathrm{Sym}_e$-conjugate to $T^{e_1}\times Y_1\times\hdots\times Y_{e_2}$, where $Y_i$ is $T$-conjugate to $\Alt_{2^{x_i}-1}$ for each $i$, and $e=e_1+e_2$. But then $|L:H|$ is a power of $2$ -- a contradiction. Suppose next that case (III) holds. Then arguing as in the $|\Lambda|=t-1$ case above we get $H$ is $\mathrm{Sym}_e$-conjugate to $T^{e_1}\times L_2$, where $L_2$ is a subdirect product of the form $L_2=\frac{1}{3^{y_0}5^{z_0}}(P_1\times\hdots\times P_{e_2})$, with $P_i$ a maximal parabolic subgroup of $T$ for all $1\le i\le e_2$, and $e_1+e_2=e$. If (III) holds with $3^{y_0}5^{z_0}=1$, then $|L:H|$ is a power of $2$ -- a contradiction.
Thus, if case (III) holds then $3^{y_0}5^{z_0}\neq 1$. This is case (4)(b) (with $r=1$) in the statement of the corollary. The proof is complete.   
\end{proof}

Corollary \ref{cor:LPoss} gives us information about the possible minimal normal subgroups $L$ (and their actions) in a minimal counterexample to Theorem \ref{thm:MinTransTheorem}. We would now like to analyse the orbit lengths of certain soluble subgroups of $L$. 

Since a point stabilizer $H$ in $L$ in this case is a direct product of subgroups $X$ which are themselves subdirect products of the form $X=\frac{1}{n}(X_1\times \hdots\times X_f)$ (see Definition \ref{Def:Subd}), our analysis will have two parts: First, we will show how to determine the orbit lengths of certain subgroups $F<L$ acting on the cosets of a subdirect product $H$ of the form $H=\frac{1}{n}(P_1\times \hdots\times P_e)$, where each $P_i$ is a subgroup of $T$. Then, we will show how to determine the orbit lengths of certain subgroups $F<L$ acting on the cosets of a subgroup $H<L$ of the form $H=H_1\times\hdots\times H_r$, where $H_i\le T^{e_i}$, and $\sum_{i=1}^{r}e_i=e$ (i.e. when $L$ has `product action' type). These two bits of information, together with Corollary \ref{cor:LPoss}, will then allow us to prove Theorem \ref{thm:MinTransTheorem}(iv). We begin with the first part:
\begin{lem}\label{SubLemma}
Let $L=G_1\times\hdots\times G_e$, where each $G_i$ is a finite group, and fix proper subgroups $P_i$, $S_i$ of $G_i$. Suppose that $H\le P_1\times\hdots\times P_e$ is a subdirect product of the form $H=\frac{1}{n}(P_1\times\hdots\times P_e)$, and that $H$ is core-free in $L$. Assume that $S:=S_1\times\hdots\times S_e$ has orbit lengths $m_1$, $\hdots$, $m_r$ in its action on the cosets of $P:=P_1\times\hdots\times P_e$ in $L$, where $\sum_i m_i=|L:P|$. Assume also that $(S_i\cap P_i^{x_i})[P_i^{x_i},P_i^{x_i}]=P_i^{x_i}$ for all $x_i\in G_i$.
Then $S$ has $r$ orbits in its action on the cosets of $H$ in $L$, with lengths $nm_1,\hdots nm_r$.
\end{lem}
\begin{proof}
Note first that since $H$ is core-free in $L$, $L$ may be viewed as a transitive permutation group acting on the set $\Omega$ of cosets of $H$ in $L$. Moreover, $H<P$, so the set of cosets of $P$ in $L$ is permutation isomorphic to a set $\Sigma$ of blocks for $P$ in $\Omega$. Note that the blocks in $\Sigma$ have size $n$, $H\unlhd P$, and $P/H$ is abelian of order $n$. It follows that $L$ is isomorphic to a subgroup of the wreath product $n\wr L^{\Sigma}$.

Now, write $\Sigma_i$, $1\le i\le r$, for the $S$-orbits in $\Sigma$, where $|\Sigma_i|=m_i$. Also, fix a block $\Delta_i\in\Sigma_i$, for each $i$. If $\Stab_S(\Delta_i)^{\Delta_i}$ is transitive for each $i$, then it follows that the $S^{\POmega}$-orbits have sizes $|\Delta||\Sigma_i|=nm_i$, for each $i$. Thus, we just need to prove that
\begin{align}\label{lab:claimF}
\text{${\Stab}_S(\Delta_i)^{\Delta_i}$ is transitive for each $i$.}
\end{align}
To this end, fix $1\le i\le e$. Then $\Stab_S(\Delta_i)^{\Delta_i}=S\cap P^{\alpha_i}$ for some $\alpha_i\in L$. To show that $\Stab_S(\Delta_i)^{\Delta_i}$ is transitive, we just need to show that $H^{\alpha_i}(S\cap P^{\alpha_i})=P^{\alpha_i}$. But $\alpha_i=(x_1,\hdots,x_e)$, for some $x_j\in G_j$. Thus, $S\cap P^{\alpha_i}=(S_1\cap P_1^{x_1})\times\hdots\times (S_e\cap P_e^{x_e})$. Since $H^{\alpha_i}$ contains $[P_1^{x_1},P_1^{x_1}]\times\hdots\times [P_e^{x_e},P_e^{x_e}]$, we have $H^{\alpha_i}(S\cap P^{\alpha_i})=P^{\alpha_i}$ by hypothesis. This proves (\ref{lab:claimF}), whence the lemma.  
\end{proof}

Next, we prove a general lemma concerning the orbits of subgroups in a transitive permutation group with `product action'.
\begin{lem}\label{lem:ProductAction}
Let $L=G_1\times\hdots\times G_e$, where each $G_i$ is a finite group, and fix proper subgroups $H_i$, $S_i$ of $G_i$. Set $S=S_1\times\hdots\times S_e$ and $H=H_1\times\hdots\times H_e$.
Then:  
\begin{enumerate}[(i)]
\item Suppose that $S_i$ acts transitively on the cosets of $H_i$ in $G_i$ for $1\le i\le e-1$, and that $S_e$ has $r$ orbits, of lengths $l_1$, $\hdots$, $l_r$ say, in its action on the cosets of $H_e$ in $G_e$. Then $S$ has $r$ orbits in its action on the cosets of $H$ in $L$, of lengths $l_1m,\hdots,l_rm$, where $m:=\prod_{i=1}^{e-1}|G_i;H_i|$.
\item Assume that there exists $0\le e_1\le e-1$ such that $S_i$ acts transitively on the cosets of $H_i$ in $G_i$ for $i\le e_1$; that $S_i$ has $2$ orbits of lengths $k_1$ and $k_2$ in each of the cosets spaces $H_i\backslash G_i$, for $e_1+1\le i\le e-1$ (in particular, $|G_i:H_i|=|G_j:H_j|=k_1+k_2$ for all $e_1+1\le i,j\le e-1$ in this case); and that $S_e$ has $r$ orbits, of lengths $l_1$, $\hdots$, $l_r$ say, in its action on the cosets of $H_e$ in $G_e$. Then $S$ has $2^{e-e_1-1}\times r$ orbits in its action on the cosets of $H$ in $G$, with $\binom{e-e_1-1}{i}$ orbits of length $l_jk_1^{i}k_2^{e-e_1-i}$ for each $0\le i\le e-e_1-1$, and each $1\le j\le r$.
\end{enumerate}
\end{lem}
\begin{proof}
The proof here is routine: if $S_i$ is a group acting on a set $\Omega_i$, then the orbits of $S_1\times\hdots\times S_e$ in its product action on $\Omega_1\times\hdots\times \Omega_e$ are precisely the sets $\Delta_{i_1,1}\times\hdots\times\Delta_{i_f,f}$, where the $\Delta_{1,j},\hdots,\Delta_{k_j,j}$ are the orbits of $S_j^{\Omega_j}$, for $1\le j\le e$. 
\end{proof}

We will also need the following corollary, which is notationally heavy, but routine.
\begin{cor}\label{cor:ProductActionCor}
Let $L$, $H$, and $S$ be as in Lemma \ref{lem:ProductAction}(i), and assume that $G_i=T$ for $1\le i\le e-1$, and $G_e=T^k$ where $T$ is a fixed nonabelian finite simple group. Assume also that $S_i=S_0$ for all $1\le i\le e-1$, and $S_e=S_0^k$ where $S_0$ is a fixed subgroup of $T$. Write $n_1,\hdots,n_u$ for the distinct orbit lengths of $S_e$ acting on the cosets of $H_e$ in $T^k$, and suppose that $S_e$ has $f(n_j)$ orbits of length $n_j$.
Let $G\le\mathrm{Sym}(\Omega)$ be a finite transitive permutation group in which $L$ is a minimal normal subgroup. Suppose that $H$ is the intersection of $L$ with a point stabilizer in $G$, and let $K$ be the kernel of the action of $G$ on the set of $L$-orbits. Let $R/K$ be a soluble subgroup of $G/K$ with orbits of lengths $a_1$, $\hdots$, $a_t$ on the set of $L$-orbits. Assume also that:
\begin{enumerate}[(i)]
\item $S_0$ is $T$-conjugate to $S_0^{\alpha}$ for all automorphisms $\alpha$ of $T$; and 
\item $N_L(S)$ is soluble (equivalently, $N_T(S_0)$ is soluble).
\end{enumerate}
Then $F$ has orbit lengths $a_iX^{(i,j)}_{k}n_jm$ for some positive partitions $f(n_j)=\sum_{k=1}^{r_{i,j}}X^{(i,j)}_k$ of $f(n_j)$, where $1\le j\le u$, $1\le i\le t$.
\end{cor}
\begin{proof}
Let $\omega\in\Omega$ such that $H=L\cap \Stab_G(\omega)$, and let $\Delta$ be the $L$-orbit containing the point $\omega$. Let $\Sigma:=\{\Delta^g\text{ : }g\in G\}$ be the set of $L$-orbits. Fix a subgroup $B$ of $G$, and assume that $B^{\Sigma}$ has orbits $\Sigma_1,\hdots,\Sigma_r$, of lengths $b_1$, $\hdots$, $b_r$, respectively. Let $\Delta_1,\hdots,\Delta_r\in\Sigma$ be representatives of these orbits. Then $B$ acts on each of the sets $\bigcup_{x\in B} \Delta_i^x$. Suppose that $\Stab_B(\Delta_i)^{\Delta_i}$ has orbits lengths $x_{i,1},\hdots,x_{i,k_i}$. Then $B^{\Omega}$ has orbit lengths $b_ix_{i,1},\hdots,b_ix_{i,k_i}$.

Now, as in the proof of Corollary \ref{Cor:FratNew}, $S_0$ being $T$-conjugate to $S_0^{\alpha}$ for all automorphisms $\alpha$ of $T$ implies that $S=S_0^e$ is $L$-conjugate to $S^{\alpha}$ for all automorphisms $\alpha\in \Aut(T)\wr\Sym_e\cong \Aut(L)$. Hence, $LN_G(S)=G$, by Lemma \ref{FrattiniArgument}. Thus, we may choose a subgroup $F$ of $N_G(S)$ containing $N_L(S)$ with $LF/L=R/L$. It follows that $F^{\Sigma}$ has orbits $\Sigma_1$, $\hdots$, $\Sigma_t$ of lengths $a_1,\hdots,a_t$, respectively. Fix $\Delta_i\in\Sigma_i$. Then since $F$ normalizes $S$ and $S$ acts trivially on $\Sigma$, each $\Stab_F(\Delta_i)^{\Delta_i}$-orbit is a union of $S$-orbits of the same length. Thus, by Lemma \ref{lem:ProductAction} the orbit lengths of $\Stab_F(\Delta_i)^{\Delta_i}$ are $X^{(i,j)}_{1}n_jm$, $\hdots$, $X^{(i,j)}_{r_i}n_jm$, where $\sum_{k=1}^{r_i} X^{(i,j)}=f(n_j)$, for each $i$, $j$. Using the first paragraph above, we deduce that $F^{\POmega}$ has orbits of lengths $a_iX^{(i,j)}_{k}n_jm$, for $1\le k\le r_i$, $1\le j\le u$, $1\le i\le t$.

Finally, since both $FL/L=R/L$ and $F\cap L=N_L(S)$ are soluble, the group $F$ is soluble. This completes the proof. 
\end{proof} 




We are now ready to exhibit orbit lengths of particular soluble subgroups of minimal counterexamples to Theorem \ref{thm:MinTransTheorem}.
\begin{lem}\label{lem:A5Lemma} Let $G$ be a minimal counterexample to Theorem \ref{thm:MinTransTheorem}, and let $L=T^e$ be a minimal normal subgroup of $G$, where $T=\Alt_5$. Let ${H}$ be the intersection of $L$ with a point stabilizer for $G$, and let $a$ be the number of $L$-orbits. 
Then $G$ has a soluble subgroup $F$ such that one of the following holds:
\begin{enumerate}[(i)]
    \item If $H$ is as in case (1)(a) of Corollary \ref{cor:LPoss}, then $F$ has orbit lengths $2^{b_1}5aX_i$, $2^{b_2}50aY_j$, where $$(\sum_i X_i,\sum_j Y_j,b_1,b_2)\in\{(4,2,0,0),(4,2,1,1),(2,2,1,0),(2,2,0,-1),(2,1,0,0),(1,1,0,-1)\}.$$
    \item If $H$ is as in case (1)(b) of Corollary \ref{cor:LPoss}, and $G/L$ is soluble, then $F$ either has orbit lengths $2^b5aX_i$, where
    $$(\sum_i X_i,b)\in\{(2,1),(2,0)\},$$
    or $5aX_i$, $10aY_j$, where $(\sum_i X_i,\sum_j Y_j)=(2,2)$.
    \item If $H$ is as in case (1)(c) of Corollary \ref{cor:LPoss}, then $F$ has orbit lengths $5aX_i$, $25aY_i$, where 
    $$(\sum_i X_i,\sum_j Y_j)=(4,4).$$
    \item If $H$ is as in case (1)(d) of Corollary \ref{cor:LPoss}, then $F$ has precisely two orbits, of lengths $10a$ and $50a$.
\end{enumerate}
\end{lem}
\begin{proof}
Set $S_0:=D_{10}$, $S:=S_0^e$, and $F:=N_G(S)$. We note first that $N_T(S_0)$ is soluble, and $D_{10}$ is $\Alt_5$-conjugate to $S_0^{\alpha}$ for all $\alpha\in\Sym_5$. Thus, (i) and (ii) in Corollary \ref{cor:ProductActionCor} hold. We will also show that $G$ modulo the kernel $K$ of the action of $G$ on the set of $L$-orbits is soluble in each case, so we can take $R=G$ in Corollary \ref{cor:ProductActionCor}. 
Since either $H$ is $\mathrm{Sym}_e$-conjugate to $T^{e-2}\times J$ with $J\le T^2$, or $H$ is $\mathrm{Sym}_e$-conjugate to $T^{e-1}\times J$ with $J\le T$, we therefore just need to compute the orbit lengths for $S_0^2$ [respectively $S_0$] in the case $e=2$ [resp. $e=1$] and then apply Lemma \ref{lem:ProductAction} and Corollary \ref{cor:ProductActionCor}.

Suppose first that we are in case (i). Then either $|T_i:H\pi_{e-1}|=2^{x_1}3$, and $|T_i:H\pi_{e}|=2^{x_2}5$, or vice versa, where $x_1+x_2=x$. Thus, $G/K$, being a minimal transitive group of $2$-power degree, is soluble (so we can indeed take $G=R$ in Corollary \ref{cor:ProductActionCor}). Moreover, $H_{e-1}\in\{C_5,D_{10}\}$ and $H_e\in\{\Alt_3,\Sym_3,\Alt_4\}$. The orbit lengths of $S_0$ acting on the cosets of $H_{e-1}$ and $H_e$ are then given in Table \ref{tab:A5}.
\begin{table}[]
    \centering
    \begin{tabular}{c|c}
      $J$ & Orbit lengths for $S_0=D_{10}<\Alt_5$ acting on the cosets of $J<\Alt_5$ \\
      \hline
        $\Alt_3$ & two of length $10$\\
        $\Sym_3$ & two of length $5$\\
        $\Alt_4$ & one of length $5$\\
        $C_5$   & two of length $1$, one of length $10$\\
        $D_{10}$ & one of length $1$, one of length $5$\\
        \hline
    \end{tabular}
    \vspace{0.1cm}
    \caption{Orbit lengths for $S=D_{10}<\Alt_5$ acting on the cosets of a subgroup $J<\Alt_5$}
    \label{tab:A5}
\end{table}
If $S_0$ has orbit lengths $c_1,\hdots,c_k$ on the cosets of $H_{e-1}$, and $d_1,\hdots,d_l$ on the cosets of $H_e$, then $S_0^2$ has orbit lengths $c_id_j$, $1\le i\le k$, $1\le j\le l$, on the cosets of $H_{e-1}\times H_{e}$. Thus, we can use the table above and Lemma \ref{lem:ProductAction} and Corollary \ref{cor:ProductActionCor} to deduce the orbit lengths of the soluble group $F=N_G(S)$.
For example, if $H_{e-1}=\Alt_3$ and $H_e=C_5$, then $S_0^2$ has four orbits of length $10$, and two orbits of length $100$ in its action on the cosets of $H_{e-1}\times H_e$ in $T^2$. By Lemma \ref{lem:ProductAction}(i), $S_0^e$ then has orbits of the same lengths in its action on the cosets of $H$ in $L$ ($m=1$ in this case). Thus, by Corollary \ref{cor:ProductActionCor}, $F$ has orbit lengths $10aX_i$, $100aY_j$, where $X_i,Y_j\in\mathbb{N}$, and $\sum_i X_i=4$, $\sum_j Y_j=2$.

We now move on to cases (ii), (iii) and (iv). In case (ii), $K\geq L$, so $G/K$ is soluble by hypothesis. In cases (iii) and (iv), $|L:H|=2^x15$, so $G/K$ is a minimal transitive group of $2$-power degree, whence soluble. Thus, we can take again $G=R$ (and hence $t=1$, $a=a_1$) in Corollary \ref{cor:ProductActionCor}. 

For part (ii), the computations of the possible orbit lengths is completely analogous to the computations in case (i) above (except that the cases $H_e\in\{\Alt_4,C_5,D_{10}\}$ cannot occur, and we also need to compute the orbit lengths of the subgroups of $\Alt_5$ of orders $1$, $2$, and $4$). 
For part (iii), one can use \Magma \cite{MR1484478} for example, to see that $S_0^2$ has four orbits of length $5$ and four orbits of length $25$, in its action on the cosets of  $J=\frac{1}{2}(\Sym_3\times D_{10})$ in $T^2$. For part (iv), $S_0^2$ has one orbit of length $10$ and one orbit of length $50$ in its action on the cosets of any diagonal subgroup in $T^2$.
We now apply Corollary \ref{cor:ProductActionCor} in each case to complete the proof.
\end{proof}

\begin{lem}\label{lem:A6Lemma} Let $G$ be a minimal counterexample to Theorem \ref{thm:MinTransTheorem}, and let $L=T^e$ be a minimal normal subgroup of $G$, where $T=\Alt_6$. Let ${H}$ be the intersection of $L$ with a point stabilizer for $G$, and let $a$ be the number of $L$-orbits. 
Then $G$ has a soluble subgroup $F$ such that one of the following holds:
\begin{enumerate}[(i)]
    \item If $H$ is as in case (1)(a) of Corollary \ref{cor:LPoss}, then $F$ has orbit lengths $2^{b}5aX_i$, $2^{b}50aY_j$, where $$(\sum_i X_i,\sum_j Y_j,b)\in\{(4,4,1),(4,4,0),(2,2,0)\}.$$
    \item If $H$ is as in case (1)(b) of  Corollary \ref{cor:LPoss}, and $G/L$ is soluble, then $F$ either has orbit lengths $2^b5aX_i$, where
    $$(\sum_i X_i,b)\in\{(4,1),(12,1),(4,0),(3,0),(2,0)\},$$
    or $5aX_i$, $10aY_j$, where $(\sum_i X_i,\sum_j Y_j)\in\{(4,4),(2,2)\}$.
\end{enumerate}
\end{lem}
\begin{proof}
Set $S_0:=D_{10}<\Alt_5<\Alt_6$, $S:=S_0^e$, and $F:=N_G(S)$. Then $S_0$ has soluble normalizer in $\Alt_6$, and $D_{10}$ is $\Alt_6$-conjugate to $S_0^{\alpha}$ for all $\alpha\in\Aut(\Alt_6)$. Thus, (i) and (ii) in Corollary \ref{cor:ProductActionCor} hold. Note also that, as in the proof of Lemma \ref{lem:A5Lemma}, $G$ modulo the kernel $K$ of the action of $G$ on the set of $L$-orbits is soluble in each of (i) and (ii). Thus, we can take $R=G$ in Corollary \ref{cor:ProductActionCor} (and hence $t=1$, $a=a_1$).
Since either $H$ is $\mathrm{Sym}_e$-conjugate to $T^{e-2}\times J$ with $J=H_{e-1}\times H_e\le T^2$, or $H$ is $\mathrm{Sym}_e$-conjugate to $T^{e-1}\times J$ with $J=H_e\le T$, we therefore just need to compute the orbit lengths for $S_0^2$ [respectively $S_0$] in the case $e=2$ [resp. $e=1$] and then apply Lemma \ref{lem:ProductAction} and Corollary \ref{cor:ProductActionCor}.

Suppose first that we are in case (i). Then, without loss of generality, we may assume that $|T_i:H\pi_{e-1}|=2^{x_1}3$ and $|T_i:H\pi_{e}|=2^{x_2}5$, where $x_1+x_2=x$. Hence, $H_{e-1}$ is in one of the two conjugacy classes of $\Alt_5$ in $\Alt_6$, while $H_e\in\{\Alt_3\times\Alt_3,(\Alt_3\times\Alt_3).2,(\Alt_3\times\Alt_3).2.2\}$. The orbit lengths of $S_0$ acting on the cosets of $H_{e-1}$ and $H_e$ are then given in Table \ref{tab:A6}.
\begin{table}[]
    \centering
    \begin{tabular}{c|c}
      $J$ & Orbit lengths for $S=D_{10}<\Alt_6$ acting on the cosets of $J<\Alt_6$ \\
      \hline
        $\Alt_3\times\Alt_3$ & four of length $10$\\
        $(\Alt_3\times\Alt_3).2$ & four of length $5$\\
        $(\Alt_3\times\Alt_3).2.2$ & two of length $5$\\
        $\Alt_5$ (transitive)   & one of length $1$, one of length $5$\\
        $\Alt_5$ (intransitive) & one of length $1$, one of length $5$\\
        \hline
    \end{tabular}
    \vspace{0.1cm}
    \caption{Orbit lengths for $S=D_{10}<\Alt_6$ acting on the cosets of a subgroup $J<\Alt_6$}
    \label{tab:A6}
\end{table}
The computations of the possible orbit lengths are then completely analogous to the computations in the proof of Lemma \ref{lem:A5Lemma}.

If we are in case (ii), then $H_e$ is either one of the groups $J$ in Table \ref{tab:A6}, with $J$ intransitive, or $|H_e|\in\{3,6,12,24\}$. We can quickly compute the orbit lengths of $S_0$ in these latter cases, and part (ii) follows as in the proof of Lemma \ref{lem:A5Lemma}.
\end{proof}

\begin{lem}\label{lem:OmLemma} Let $G$ be a minimal counterexample to Theorem \ref{thm:MinTransTheorem}, and let $L=T^e$ be a minimal normal subgroup of $G$, where $T\in\{\PSp_4(4),\POmega^+_8(2)\}$. Let ${H}$ be the intersection of $L$ with a point stabilizer for $G$, and let $a$ be the number of $L$-orbits. 
Then $G$ has a soluble subgroup $F$ such that either $n=120a$ and $F$ has two orbits of lengths $24a$ and $96a$; or $n=960a$ and $F$ has two orbits of length $192a$ and $768a$.
\end{lem}
\begin{proof}
In these cases, $H$ is $\mathrm{Sym}_e$-conjugate to $T^{e-1}\times H_e$, where $(T,H_e)$ is one of the pairs $(\PSp_4(4),\PSp_2(16))\text{ (}2 \text{classes), }(\POmega^+_8(2),\Alt_9)\text{ (}3\text{ classes), or }(\POmega^+_8(2),\POmega_7(2))\text{ (}3\text{ classes)}$. Thus, $|T:H|$ is $120$, $960$, or $120$, respectively. If $T=\PSp_4(4)$, then let $S_0$ be the maximal parabolic subgroup of $T$ of index $425$. If $T=\POmega^+_8(2)$, then let $S_0$ be the maximal parabolic subgroup of $T$ of index $1575$. Then in each case, $N_T(S_0)$ is soluble, and $S_0$ and $S_0^{\alpha}$ are $T$-conjugate for all $\alpha$ in $\Aut(T)$. That is, (i) and (ii) in Corollary \ref{cor:ProductActionCor} hold. Moreover, $S_0$ has two orbits in its action on the cosets of $H_e$ in $T$, of lengths $24$ and $96$ when $|T:H|=120$, and lengths $192$ and $768$ when $|T:H|=960$. 

Let $\Sigma$ be the set of $L$-orbits. Then since $15$ divides $|L:H|$ in each case, $G^{\Sigma}$ is a $2$-group. Thus, as in the proof of Lemmas \ref{lem:A5Lemma} and \ref{lem:A6Lemma}, we see that either $n=120a$ and $F$ has two orbits of lengths $24a$ and $96a$, or $n=960a$ and $F$ has two orbits of length $192a$ and $768a$.  
\end{proof}

In the next three lemmas we deal with the cases $T\in\{\Alt_8,\Alt_{16},\Lg_2(q)\text{ : }q=2^{x'}3^{y'}5^{z'}-1\}$. Apart from the case $q=2^{x'}3^{y'}5^{z'}-1$ with $\{y',z'\}=\{1,0\}$, we will find a soluble subgroup $S$ of $L$ with convenient orbit lengths, and then simply set $F:=S$ (we do not need to trouble to find a soluble subgroup with fewer orbits, as we did in the $T\in\{\Alt_5,\Alt_6,\PSp_4(4),\POmega^+_8(2)\}$ cases).
\begin{lem}\label{lem:A8Lemma} Let $G$ be a minimal counterexample to Theorem \ref{thm:MinTransTheorem}, and let $L=T^e$ be a minimal normal subgroup of $G$, where $T=\Alt_8$. Let ${H}$ be the intersection of $L$ with a point stabilizer for $G$, so that ${H}$ is of type (2) in Corollary \ref{cor:LPoss}. That is, $H$ is $\mathrm{Sym}_e$-conjugate to $T^{e_1}\times Y_1\times\hdots\times Y_{e_2}\times H_e$, where $e_1+e_2=e-1$; $Y_i$ is $T$-conjugate to $\Alt_7$ for each $i$; and $H_e$ has shape $7:3$, $2^3:(7:3)$, $\GL_3(2)$, or $2^3:\GL_3(2)$.
Let $a$ be the number of $L$-orbits, let $S_0:=7:3<T$, and set $F:=S_0^e<L$. Then
\begin{enumerate}[(i)]
\item Suppose that $H_e=7:3$, and define $(s_1,l_1):=(1,1)$; $(s_2,l_2):=(11,7)$; and $(s_3,l_3):=(42,21)$. Then $F$ has $as_j\binom{e_2}{i}$ orbits of size $7^{i}l_j$, for each $1\le i\le e_2$, $1\le j\le 3$.
\item Suppose that $H_e\cong 2^3:(7:3)$ or $K\cong \GL_3(2)$, and define $(s_1,l_1):=(1,1)$; $(s_2,l_2):=(5,7)$; and $(s_3,l_3):=(4,21)$. Then $F$ has $as_j\binom{e_2}{i}$ orbits of size $7^{i}l_j$, for each $1\le i\le e_2$, $1\le j\le 3$.
\item Suppose that $H_e=2^3:\GL_3(2)$, and define $(s_1,l_1):=(1,1)$; and $(s_2,l_2):=(2,7)$. Then $F$ has $as_j\binom{e_2}{i}$ orbits of size $7^{i}l_j$, for each $1\le i\le e_2$, $1\le j\le 2$.
\end{enumerate}
\end{lem}
\begin{proof}
Note that $S_0$ has orbits of lengths $7$ and $1$ in its action on the cosets of $\Alt_7$ in $T$, while $S_0$ has $s_i$ orbits of length $l_i$ in its action on the cosets of $H_1$ in $T$, where $s_i$ and $l_i$ are as defined in each of the listed cases.
We then apply Lemma \ref{lem:ProductAction}(ii) to find the orbit lengths of $F$ in its action on the cosets of $H$ in $L$: we see that $F$ has $s_j\binom{e_2}{i}$ orbits of size $7^{i}l_j$ in each case.

Finally, since this holds for each $G$-conjugate of $H$ in $L$, and $F$ fixes each $L$-orbit, we deduce that $F<G$ has $as_j\binom{e_2}{i}$ orbits of size $7^{i}l_j$ in $[n]$, as required.
\end{proof}

\begin{lem}\label{lem:A16Lemma} Let $G$ be a minimal counterexample to Theorem \ref{thm:MinTransTheorem}, and let $L=T^e$ be a minimal normal subgroup of $G$, where $T=\Alt_{16}$. Let ${H}$ be the intersection of $L$ with a point stabilizer for $G$, so that ${H}$ is of type (3) in Corollary \ref{cor:LPoss}. That is, $H$ is $\mathrm{Sym}_e$-conjugate to $T^{e_1}\times Y_1\times\hdots\times Y_{e_2}\times H_e$, where $e_1+e_2=e-1$; $Y_i$ is $T$-conjugate to $\Alt_{15}$ for each $i$; and $H_e$ has shape $\Alt_{14}.2^b$, for some $b\in \{0,1\}$.
Let $a$ be the number of $L$-orbits, let $S_0:=C_{15}<T$, and set $F:=S_0^e<L$. Then $F$ has $\frac{16}{2^b}\binom{e_2}{i}a$ orbits of size $15^{i+1}$, for each $0\le i\le e_2$. 
\end{lem}
\begin{proof}
The proof is identical to the $\Alt_8$ case above: $S_0$ has orbits of lengths $15$ and $1$ in its action on the cosets of $\Alt_{15}$ in $T$, while $S_0$ has $\frac{16}{2^b}$ orbits of length $15$ in its action on the cosets of $H_e$ in $T$
We then apply Lemma \ref{lem:ProductAction} and argue as in the proof of Lemma \ref{lem:A8Lemma} to complete the proof. 
\end{proof}

\begin{lem}\label{lem:MersenneLemma} Let $G$ be a minimal counterexample to Theorem \ref{thm:MinTransTheorem}, and let $L=T^e$ be a minimal normal subgroup of $G$, where $T=\Lg_2(q)$, with $q$ an odd prime power of the form $q=2^{x'}3^{y'}5^{z'}-1$, $0\le y',z'\le 1$. Let ${H}$ be the intersection of $L$ with a point stabilizer for $G$, so that ${H}$ is of type (4) in Corollary \ref{cor:LPoss}. Let $a$ be the number of $L$-orbits, and set $F:=S_0^e<L$, where $S_0$ is a maximal parabolic subgroup of $T$. Then
\begin{enumerate}[(i)]
\item Suppose that $H$ has type (4)(a) from Corollary \ref{cor:LPoss}, so that $q=p$ is a Mersenne prime with $15$ dividing $p-1$, and $H$ is $\Sym_e$-conjugate to $T^{e_1}\times Y_1\times\hdots\times Y_{e_2}\times (p:\frac{p-1}{6})\times (p:\frac{p-1}{10})$, where each $Y_i$ is a maximal parabolic subgroup of $T$. Then $F:=S$ has $a\binom{e_2}{i}$ orbits of size $15p^{i}$; $2a\binom{e_2}{i}$ orbits of size $15p^{i+1}$; and $a\binom{e_2}{i}$ orbits of size $15p^{i+2}$, for each $0\le i\le e_2$.
\item Suppose that $H$ has type (4)(b) from Corollary \ref{cor:LPoss}, so that $q=p$ is a Mersenne prime and $H$ is $\mathrm{Sym}_e$-conjugate to $T^{e_1}\times L_1$, where $L_1$ is a subdirect product of the form $L_1=\frac{1}{t}(L_2\times H\pi_e)$; $L_2$ is a subdirect product of the form $L_2=\frac{1}{s}(P_1\times\hdots\times P_{e_2})$; each $P_i$ is a maximal parabolic subgroup of $T$; $H\pi_e$ is a subgroup of index $r$ in a maximal parabolic subgroup of $T$; $rst=3^{y_0}5^{z_0}$ is the odd part of $|L:H|$; and $e_1+e_2=e-1$. Then $F:=S$ has $a\binom{e_2+1}{i}$ orbits of size $3^{y_0}5^{z_0}p^{i}$ for each $0\le i\le e_2+1$.
\item Suppose that $H$ has type (4)(c) from Corollary \ref{cor:LPoss}, so that $H$ is $\mathrm{Sym}_e$-conjugate to $T^{e-1}\times (q:\frac{q-1}{2r})$, where $y'+z'=1$; $0\in\{y',z'\}$; and $r=3^{1-y'}5^{1-z'}>1$. Then $F:=N_G(S)$ is soluble, and has precisely two orbits, of sizes $ar$ and $aqr$.
\end{enumerate}
\end{lem}
\begin{proof} 
We need to determine the orbit lengths of $S_0$ in its action of the cosets of the subgroups $X\in\{q:\frac{q-1}{2},q:\frac{q-1}{6},q:\frac{q-1}{10}\text{ }\mathrm{(}5\mid q-1\mathrm{)},q:\frac{q-1}{15}\text{ }\mathrm{(}5\mid q-1\mathrm{)}\}$ for part (ii); while we need to determine the orbit lengths of $S_0^2$ in its action of the cosets of the subgroup $(p:\frac{p-1}{6})\times (p:\frac{p-1}{10})<T^2$ for part (i).

We first consider the orbits of the action of $S_0\cong q:\frac{q-1}{2}$ on the cosets of itself in $T$. This is equivalent to the action of $S_0$ on the set of one dimensional subspaces in $\mathbb{F}_q^2$. It is well known that $S_0$ has one orbit of size $1$ and one orbit of size $q$ in this action. It follows easily that if $X$ has shape $q:\frac{q-1}{2r}$ with $r$ a divisor of $\frac{q-1}{2}$, then $S_0$ has one orbit of size $r$, and one orbit of size $rq$. 
Next, when $p=q$, $S_0^2$ has four orbits on the cosets of $(p:\frac{p-1}{6})\times (p:\frac{p-1}{10})$, of lengths $15$, $15q$, $15q$, and $15q^2$. Lemma \ref{lem:ProductAction}(ii) then implies that $F$ has $\binom{e_2}{i}$ orbits of size $15p^{i}$; $2a\binom{e_2}{i}$ orbits of size $15p^{i+1}$; and $a\binom{e_2}{i}$ orbits of size $15p^{i+2}$ in its action on the cosest of $H$ in $L$, for each $0\le i\le e_2$. Since our arguments hold with $H$ replaced by any $G$-conjugate of $H$, we deduce that $F$ has $a\binom{e_2}{i}$ orbits of size $15p^{i}$; $2a\binom{e_2}{i}$ orbits of size $15p^{i+1}$; and $a\binom{e_2}{i}$ orbits of size $15p^{i+2}$ in $[n]$, for each $0\le i\le e_2$.

Suppose now that we are in case (ii). Note that \begin{align}\label{nt}
\text{$(S_0\cap P_i^x)[P_i^{x_i},P_i^{x_i}]=P_i^{x_i}$ for all $x_i$ in $T$.} 
\end{align}
Thus, by the arguments above, together with Lemma \ref{SubLemma}, we deduce that $S_0^{e_2}$ has $\binom{e_2}{i}$ orbits of length $sq^i$ on the cosets of $L_2$ in $T^{e_2}$, for each $0\le i\le e_2$. Hence, since $S_0$ has two orbits, of sizes $r$ and $qr$, on the cosets of $H\pi_e$ in $T$, we deduce from Lemma \ref{lem:ProductAction} that $S_0^{e_2+1}$ has $\binom{e_2}{i}$ orbits of length $srq^i$ and $\binom{e_2}{i}$ orbits of length $srq^{i+1}$ in its action on the cosets of $L_2\times H\pi_e$ in $T^{e_2+1}$. By (\ref{nt}), we can then use Lemma \ref{SubLemma} again: we get that $S_0^{e_2+1}$ has $\binom{e_2}{i}$ orbits of length $rstq^i$ and $\binom{e_2}{i}$ orbits of length $rstq^{i+1}$ in its action on the cosets of $L_1$ in $T^{e_2+1}$. Note that $rst=3^{x_0}5^{x_0}$ and $\binom{e_2}{i}+\binom{e_2}{i-1}=\binom{e_2+1}{i}$ for $i>0$. We can then deduce from Lemma \ref{lem:ProductAction}(i) (with $m=1$) that $F$ has $\binom{e_2+1}{i}$ orbits of size $3^{y_0}5^{z_0}p^{i}$ in its action on the cosets of $H$ in $L$, for each $0\le i\le e_2+1$. Since our arguments are independent of the choice of $G$-conjugate of $H$ in $L$, we see as above that $F$ has
$a\binom{e_2+1}{i}$ orbits of size $3^{y_0}5^{z_0}p^{i}$ in $[n]$, for each $0\le i\le e_2+1$.

Finally, the proof in case (iii) is identical to the proof of Lemma \ref{lem:OmLemma}. 
\end{proof}

\begin{lem}\label{lem:A5A6Lemma} Let $G$ be a minimal counterexample to Theorem \ref{thm:MinTransTheorem}, and let $L=T^e$ be a minimal normal subgroup of $G$, where $T=\Alt_m$ with $m\in\{5,6\}$, and $G/L$ is insoluble. Let ${H}$ be the intersection of $L$ with a point stabilizer for $G$, and let $a$ be the number of $L$-orbits. Then 
\begin{enumerate}[(i)]
\item $H$ is as in case (1)(b) of Corollary \ref{cor:LPoss};
\item $|T:H_e|=2^{x_1}5$ and $a=2^{x_2}3$;
\item $G/L$ has a unique nonabelian chief factor $L'$ such that $L'\cong \Lg_2(p)^f$, where $p$ is a Mersenne prime; and
\item There exist positive integers $X$, $b$, and $e_2\le e$; positive partitions $\binom{e_2}{i}=\sum_j C^{(i)}_j$ and $X=\sum_k X^{(i,j)}_k$ of $\binom{e_2}{i}$ and $X$ (for each $i,j$); and a soluble subgroup $F$ of $G$, such that one of the following holds:
\begin{itemize}
    \item $T=\Alt_5$, $(X,b)\in\{(2,1),(2,0)\}$, and $F$ has $C^{(i)}_j$ orbits of length
    $2^b3p^i5X^{(i,j)}_k$, for $0\le i\le e_2$, and each $j,k$.
    \item $T=\Alt_6$, $(X,b)\in\{(4,1),(4,0),(2,0)\}$, and $F$ has $C^{(i)}_j$ orbits of length
    $2^b3p^i5X^{(i,j)}_k$, for $0\le i\le e_2$, and each $j,k$. 
\end{itemize}
\end{enumerate}
\end{lem}
\begin{proof}
Let $K$ be the kernel of the action of $G$ on the set $\Sigma$ of $L$-orbits, and fix an $L$-orbit $\Delta$.
Then since $G/K$ is an insoluble minimal transitive group, its degree $a$ cannot be a power of $2$. Hence, either $3$ or $5$ must divide $a$. Thus, the odd part of $|\Delta|=|L:H|$ is either $3$ or $5$, and it follows from Corollary \ref{cor:LPoss}, and by inspection of the subgroups of $\Alt_5$ of $\Alt_6$, that the only possibility is that $H$ is as in case (1)(b) of Corollary \ref{cor:LPoss}. Part (i) follows. 

Now, by Corollary \ref{cor:LPoss}(1)(b), $H_e$ is not transitive, and not equal to a natural point stabilizer $\Alt_{m-1}<T=\Alt_m$. Thus, going through the subgroups of $\Alt_5$ and $\Alt_6$, we see that $|T:H_e|$ must be of the form $|T:H_e|=2^{x_1}5$. Hence, $G/K$ is an insoluble minimal transitive group of degree $2^{x_2}3$. This proves (ii). Part (iii) then follows immediately from \cite[Theorem 1.8]{MR3812195}.

We now prove (iv). Set $S_0:=D_{10}<T$, and $S:=S_0^e$. Then as shown in Lemmas \ref{lem:A5Lemma} and \ref{lem:A6Lemma}, the hypotheses in Corollary \ref{cor:ProductActionCor} hold with this choice of $S_0$. Thus we can, and do, apply Corollary \ref{cor:ProductActionCor}, with $R/K$ chosen to be the soluble subgroup of $G/K$ exhibited in Lemma \ref{lem:MersenneLemma}(ii), with $p=2^{x_1}-1$ and $rst=3$ .
We can then deduce (iv) by taking the possible orbit lengths of $S_0$ acting on the cosets of $H_e$ from Tables \ref{tab:A5} and \ref{tab:A6} (noting again that $H_e$ is intransitive, and that $H_e\neq \Alt_4$ in the $\Alt_5$ case). 
\end{proof}

Finally, we are ready to prove Theorem \ref{thm:MinTransTheorem}.
\begin{proof}[Proof of Theorem \ref{thm:MinTransTheorem}]
Let $G$ be a counterexample to the theorem of minimal degree, and let $L\cong T^e$ be a minimal normal subgroup of $G$, with $T$ simple. Then $T$ is nonabelian, by Lemma \ref{lem:RedLemma}. Let $H$ be the intersection of $L$ with a point stabilizer in $G$, and write $|L:H|=2^{x_0}3^{y_0}5^{z_0}$, where $x_0\le x$, $y_0\le y$, and $z_0\le z$. Thus, if $K$ denotes the kernel of the action of $G$ on the set $\Sigma$ of $L$-orbits, then $G^{\Sigma}\cong G/K$ is minimal transitive of degree $2^{x-x_0}3^{y-y_0}5^{z-z_0}$. Also, either $y_0$ or $z_0$ is non-zero, by Corollary \ref{Cor:FratNew}(i). Thus, since $K/L$ is soluble by Lemma \ref{lem:RedLemma}, the minimality of $G$ as a counterexample implies that $G$ has at most $1+y-y_0+z-z_0\le y+z$ nonabelian chief factors. That is, (i) holds.

Now, $T$ is one of the groups in part (ii) of Theorem \ref{thm:MinTransTheorem}, by Corollary \ref{cor:LPoss}, so (ii) also holds. 
Next, we prove (iii). So assume that $G$ has two nonabelian chief factors, and let $L'\cong T'^{e'}$ be a nonabelian chief factor of $G/L$. Then the minimal transitive group $G^{\Sigma}\cong G/K$ has $L'$ as a chief factor, since $K/L$ is soluble. 
Since minimal transitive groups of $2$-power degree are $2$-groups, we deduce that $G^{\Sigma}\cong G/K$ is minimal transitive of degree either (1) $2^{x-x_0}5$ or (2) $2^{x-x_0}3$. The size of an $L$-orbit in these cases is $2^{x_0}3$ or $2^{x_0}5$, respectively. Now, by Corollary \ref{cor:LPoss}, the only possibility for $T$ when the size of an $L$-orbit is $2^{x_0}3$ is $T=\Lg_2(p)$ with $p$ a Mersenne prime; while the only possibility for $T$ when the size of an $L$-orbit is $2^{x_0}5$ is $T\in\{\Alt_5,\Alt_6,\Lg_2(p)\text{ : } p \text{ a Mersenne prime}\}$. This gives us what we need. 
Moreover, inspection of Table \ref{tab:OrbitSizes} and the inductive hypothesis imply that in case (1), we must have $T'\in\{\Alt_5,\Alt_6,\Lg_2(p')\text{ : } p \text{ a Mersenne prime}\}$; while in case (2), we have $T'=\Lg_2(p')$ with $p'$ a Mersenne prime. This proves (iii).

Finally, in Lemmas \ref{lem:A5Lemma}, \ref{lem:A6Lemma}, \ref{lem:OmLemma}, \ref{lem:A8Lemma}, \ref{lem:A16Lemma}, \ref{lem:MersenneLemma}, and Lemma \ref{lem:A5A6Lemma}, we exhibit a soluble subgroup $F$ of $G$ with particular orbit lengths. These orbit lengths are as listed in Table \ref{tab:OrbitSizes}, and this completes the proof.
\end{proof}

\section{The proof of Theorem \ref{thm:transgens}}\label{sec:proof}
In this section, we prove Theorem \ref{thm:transgens}.
First, we require the following definitions and lemma from \cite{MR3812195}.

\begin{Definition}\label{KDef} For a positive integer $s$ with prime factorisation $s=p_{1}^{r_{1}}p_{2}^{r_{2}}\hdots p_{t}^{r_{t}}$, set\\ $\omega{(s)}:=\sum r_{i}$, $\omega_{1}{(s)}:=\sum r_{i}p_{i}$, $K(s):=\omega_{1}{(s)}-\omega(s)=\sum r_i(p_i-1)$ and 
$$\ws(s)=\frac{s}{2^{K(s)}}\binom{K(s)}{\left\lfloor\frac{K(s)}{2}\right\rfloor}.$$ \end{Definition} 
\noindent For a prime $p$, write $s_{p}$ for the $p$-part of $s$. 
\begin{Definition} Let $p$ be prime, and let $s$ be a positive integer. We define
$$E_{sol}(s,p):=\min\left\{\ws(s),s_{p}\right\}.$$
 \end{Definition}

\begin{lem}\cite[Corollary 4.25(ii)(b)]{MR3812195}\label{chiefCor}
Let $G$ be a transitive permutation group of degree $n$, and suppose that $G$ is imprimitive with a minimal block of size $2$. Let $\Sigma$ be the associated system of blocks of size $\frac{n}{2}$, so that $G$ may be viewed as a subgroup of the wreath product $2\wr G^{\Sigma}$. Let $F$ be a soluble subgroup of $G^{\Sigma}$, and suppose that $F$ has orbit lengths $s_1,\hdots,s_r$. 
Then $$d(G)\le \sum_{i=1}^rE_{sol}(s_i,2)+d(G^{\Sigma}).$$
\end{lem}

\begin{cor}\label{ChiefCorCor}
Let $G$ be a transitive permutation group of degree $n$, and suppose that $G$ is imprimitive with a minimal block of size $2$. Let $\Sigma$ be the associated system of blocks of size $\frac{n}{2}$, so that $G$ may be viewed as a subgroup of the wreath product $2\wr G^{\Sigma}$. Suppose that $\frac{n}{2}$ has the form $\frac{n}{2}=2^x3^y5^z$, with $0\le y,z\le 1$, and that $G^{\Sigma}$ contains no soluble transitive subgroups. Let $F$ be a soluble subgroup of $G^{\Sigma}$ with orbit lengths as exhibited in Table \ref{tab:OrbitSizes}.
\begin{enumerate}[(1)]
    \item If $F$ is as in row 12 of Table \ref{tab:OrbitSizes}, then $d(G)\le 2^{e_2+b+1}a+d(G^{\Sigma})$.
    \item If $F$ is as in row 22 of Table \ref{tab:OrbitSizes}, then $d(G)\le 2^{e_2+b+2}a+d(G^{\Sigma})$.
    \item If $F$ is as in row 24, 25, or 26 of Table \ref{tab:OrbitSizes}, then $d(G)\le 2^{e_2}54a+d(G^{\Sigma})$; $d(G)\le 2^{e_2}10a+d(G^{\Sigma})$; or $d(G)\le 2^{e_2}3a+d(G^{\Sigma})$, respectively.
    \item If $F$ is as in row 27 of Table \ref{tab:OrbitSizes}, then $d(G)\le 2^{e_2+4-b}a+d(G^{\Sigma})$.
    \item If $F$ is as in row 28 of Table \ref{tab:OrbitSizes}, then $d(G)\le 2^{e_2}3a+d(G^{\Sigma})$.
    \item If $F$ is as in row 29 of Table \ref{tab:OrbitSizes}, then $d(G)\le 2^{e_2}a+d(G^{\Sigma})$.
    \item If $F$ is as in row 30 of Table \ref{tab:OrbitSizes}, then $d(G)\le 2a+d(G^{\Sigma})$.
\end{enumerate}
\end{cor}
\begin{proof}
Write $s_1$, $\hdots$ $s_r$ for the $F$-orbit lengths. Then by Lemma \ref{chiefCor} and the definition of $E_{sol}$, we have 
\begin{align}\label{fgt}
    d(G)\le \sum_{i=1}^r (s_i)_2 +d(G^{\Sigma}).
\end{align}
Parts (3), (4), (5), and (6) now follow immediately from (\ref{fgt}) and inspection of the relevant row in Table \ref{tab:OrbitSizes}. 

Suppose now that $F$ is as in row 12 or 22 of Table \ref{tab:OrbitSizes}. Then (\ref{fgt}) implies that 
\begin{align*}
   d(G)\le 2^b\sum_{i,j,k}C^{(i)}_j (X^{(i,j)}_k)_2 +d(G^{\Sigma}). 
\end{align*}
Since $(X^{(i,j)}_k)$ is an ordered positive partition of either $2$ or $4$, the result now follows easily by going through each of the possibilities for $(X^{(i,j)}_k)_k$, and noting that $\sum_{i,j}C^{(i)}_j=\sum_i\binom{e_2}{i}=2^{e_2}$.
\end{proof}

We are now ready to prove Theorem \ref{thm:transgens}.
\begin{proof}[Proof of Theorem \ref{thm:transgens}]
By \cite[Theorem 1.1]{MR3812195}, we may assume that $n$ has the form $n=2^x3^y5$, where $y\in\{0,1\}$. Furthermore, we may assume that $17\le x\le 26$ if $y=0$, and $15\le x\le 35$ if $y=1$; that $G$ has
a block system $\Sigma$ consisting of blocks of size $2$; and that $G^{\Sigma}$ contains no soluble transitive subgroup.

It then follows from Theorem \ref{thm:MinTransTheorem} that $G^{\Sigma}$ has a soluble subgroup $F$ whose orbit lengths occur in one of the rows of Table \ref{tab:OrbitSizes}. We can then use this fact, together with Lemma \ref{chiefCor} to prove the theorem. Indeed, suppose that we have proved the theorem for all degrees properly dividing $n$, and that the possible orbit lengths for the soluble subgroup $F$ are $s_{1,j},s_{2,j},\hdots,s_{r_j,j}$, for $1\le j\le t$. We then compute the maximum value of  
$\sum_{i=1}^{r_j}E_{sol}(s_{i,j},2)$ over $1\le j\le t$, and add it to our previously obtained bound for $d(G^{\Sigma})$. By Lemma \ref{chiefCor}, this gives a bound for $d(G)$.

Let us illustrate this strategy with some examples. Suppose first that $n=2^{17}5$. From Table \ref{tab:OrbitSizes}, we see that $G^{\Sigma}\le\Sym_{2^{16}5}$ has a soluble subgroup whose orbit lengths lie in one of the rows 7, 18, 19, or 29 (with $y_0=0$, $z_0=1$ in the last case). Computing as in the $n=2^{15}15$ example given at the beginning of the section, we see that the lengths of the $F$-orbits are one of the following:
\begin{itemize}
    \item $2^i$ orbits of equal lengths, $i\in\{1,2\}$;
    \item $2$ orbits of length $2^{14}5$ and one orbit of length $2^{15}5$; 
    \item $1$ orbit of length $2^{14}5$ and $1$ orbit of length $2^{14}15$; or
    \item $\binom{e_2}{i}a$ orbits of size $5p^i$ and $\binom{e_2}{i}a$ orbits of size $5p^{i+1}$, for each $0\le i\le e_2$, where the possibilities for $e_2$ and $p$ are $(e_2,p)\in\{(A,2^5-1),(1,2^{13}-1)\text{ : }1\le A\le 2\}$, and $\frac{n}{2}=5a(p+1)^{e_2}$.
\end{itemize}
We make some remarks about the last point above. Since $(p+1)^{e_2}$ must divide $\frac{n}{2}$, we only need to consider the Mersenne primes $p=2^u-1$ such that $p+1=2^u<2^k:=(\frac{n}{2})_2$, and those $e_2$ with $1\le e_2\le \left\lfloor \frac{k}{u}\right\rfloor-1$. And since we are assuming, throughout this proof, that $n_2\le 2^{35}$, the only Mersenne primes under consideration are $2^u-1$, for $u$ a prime less than or equal to $31$. 
In the case $n=2^{17}5$ above, we must also have that $5$ divides $p-1$. So only the cases $u\in\{5,13\}$ can occur.

We now proceed to compute a bound for $d(G)$ (for $n=2^{17}5$) in each case, by using Table \ref{tab:OrbitSizes} and Corollary \ref{chiefCor}, as explained above. By \cite[Table A.1]{MR3812195}, we have $d(G^{\Sigma})\le 65538$.
If $F$ has $2$ orbits of length $2^{15}5$ then we get $d(G)\le 2E_{sol}(2^{15}5,2)+65538=123274$. Similarly, if $F$ has $4$ orbits of length $2^{14}5$ then we get $d(G)\le 126313$; if $F$ has $2$ orbits of length $2^{14}5$ and $1$ orbit of length  $2^{15}5$ then $d(G)\le 124793$; and if $F$ has $1$ orbit of length $2^{14}5$ and $1$ orbit of length  $2^{14}15$ then $d(G)\le 97115$. Suppose now that $F$ has $\binom{e_2}{i}d$ orbits of size $5p^i$ and $\binom{e_2}{i}d$ orbits of size $5p^{i+1}$, for each $0\le i\le e_2$, where $(e_2,p)\in\{(1,2^5-1),(2,2^5-1),(1,2^{13}-1)\}$ and $a=\frac{n}{10(p+1)^{e_2}}$. If $(e_2,p)=(1,2^5-1)$, then $F$ has $2^{11}$ orbits of size $5$, and $2^{11}$ orbits of size $5\times 31=155$. Therefore, we get $d(G)\le  2^{11}E_{sol}(5,2)+2^{11}E_{sol}(155,2)+65538=69634$.
The other cases are entirely similar, and we see that $d(G)$ is always bounded above by $126313$. Since $\left\lfloor\frac{c2^{17}5}{\sqrt{\log{2^{17}5}}}\right\rfloor=129117$, this proves Theorem \ref{thm:transgens} in the case $n=2^{17}5$.

Next, we will give an example to show how the required upper bounds can be derived in the cases $n=2^x15$. More precisely, we will go through the necessary computations in the case $n:=2^{15}15$. In order to do this, we will first need an upper bound on $d(G^{\Sigma})$. Since $G^{\Sigma}$ is a transitive permutation group of degree $2^{14}15$, we can deduce from \cite[Table A.1]{MR3812195} that $d(G^{\Sigma})\le 49156$. Using the classification of the transitive groups of degree $48$ from \cite{HRT}, however, we can improve this bound. In order to avoid repetition of details, we will simply refer to \cite[proof of Lemma 5.12, case 2, page 40]{MR3812195}, where the proofs of the bounds in \cite[Table A.1]{MR3812195} are given. The author uses the classification of the transitive permutation groups of degree up to $32$ \cite{CH}, together with the bounds in Lemma \ref{chiefCor} above, to derive the upper bound on $d(G^{\Sigma})\le 49156$. Within this computation, the estimate $d(X)\le 16$ for transitive groups $X\le \Sym_{48}$ is used. We can now use the exact bound $d(X)\le 10$ for transitive $X\le \Sym_{48}$ from \cite[Section 5]{HRT}. Performing the computations in \cite[proof of Lemma 5.12, case 2, page 40]{MR3812195} again with this new bound, we quickly deduce that
\begin{align}\label{newb48}
  d(G^{\Sigma})\le 49150\text{ whenever }G^{\Sigma}\le{\Sym}_{2^{14}15}\text{ is transitive.}  
\end{align}

Using (\ref{newb48}) and Lemma \ref{chiefCor}, we can now determine an upper bound for $d(G)$ when $n=2^{15}15$. At the beginning of the section, we listed the possible $F$-orbit lengths in the case when $G$ has a unique nonabelian chief factor, and this chief factor is a direct product of copies of $\Alt_5$ (these cases correspond to rows 1--11 in Table \ref{tab:OrbitSizes}). Entirely analogous calculations give us the possible $F$-orbit lengths when $G$ has a unique nonabelian chief factor, and this chief factor is a direct product of copies of $\Alt_6$ (these cases correspond to rows 13--21 in Table \ref{tab:OrbitSizes}) or $\PSp_4(4)$ or $\POmega^+_8(2)$ (row 23). When $F$ has orbit lengths as in rows 12, 22, and 24--30 of \ref{tab:OrbitSizes}, we use Corollary \ref{ChiefCorCor} to bound $d(G)$. Taking the maximum of these bounds for $d(G)$, we get $d(G)\le 97401$ in each case. This proves the theorem in this case, since $\left\lfloor\frac{c2^{15}15}{\sqrt{\log{2^{15}15}}}\right\rfloor=97895$.

The remaining cases with $n$ of the form $n=2^{x}15$, and $16\le x\le 35$, are entirely similar, except that we use the bound $d(G^{\Sigma})\le \left\lfloor\frac{c2^{x-1}15}{\sqrt{\log{2^{x-1}15}}}\right\rfloor$ for $d(G^{\Sigma})$ (which holds by the inductive hypothesis). We note in particular that we get $d(G^{\Sigma})\le 189053$ when $x=16$ (this will be important). The result now follows in each case, except when $x=17$. So assume that $x=17$. By \cite[Theorem 1.1]{MR3812195}, we may assume that $G^{\Sigma}$ is imprimitive with a minimal block of size $2$. Let $K_1$ be the kernel of the action of $G^{\Sigma}$ on a set $\Sigma_1$ of blocks of size $2$ in $\Sigma$.
Assume first that $G^{\Sigma}\cap K_1\neq 1$. If $(G^{\Sigma})^{\Sigma_1}\cong G^{\Sigma}/K_1$ contains a soluble transitive subgroup $S/K_1$, then $S$ is a soluble transitive subgroup of $G^{\Sigma}$ -- a contradiction. Thus, $(G^{\Sigma})^{\Sigma_1}$ contains no soluble transitive subgroups. Because of this, we can use our previously obtained bound $d(G^{\Sigma})\le 189053$ in this case. If $G^{\Sigma}\cap K_1=1$, then $d(G^{\Sigma})=d(G^{\Sigma_1})\le \left\lfloor\frac{c2^{15}15}{\sqrt{\log{2^{15}15}}}\right\rfloor=98547$, so in either case, we get $d(G^{\Sigma})\le 189053$. By computing the upper bound for $d(G)$ as we did in the case $n=2^{15}15$ above, we now get $d(G)\le 371369$. This gives us what we need, since $\left\lfloor\frac{c2^{17}15}{\sqrt{\log{2^{17}15}}}\right\rfloor=372380$.
\end{proof}
\bibliographystyle{plain}

\end{document}